\documentclass[11pt]{amsart}
\usepackage{graphicx} 
\usepackage{amsthm,amsmath,mathrsfs}
\usepackage[margin=1in, footskip=0.25in]{geometry}
\usepackage{amssymb}
\usepackage[latin1]{inputenc}
\usepackage{tikz}
\usetikzlibrary{trees}
\usepackage{url}
\usepackage{color}

\newtheorem{theorem}{Theorem}[section]

\newtheorem{lemma}[theorem]{Lemma}
\newtheorem{proposition}[theorem]{Proposition}
\newtheorem{remark}[theorem]{Remark}
\newtheorem{construction}[theorem]{Construction}
\newtheorem{definition}[theorem]{Definition}

\newtheorem{problem}{Problem}

\newcommand{\Soc}{\mathrm{Soc}}

\newcommand{\Sz}{\mathrm{Sz}}

\newcommand{\Aut}{\mathrm{Aut}}
\newcommand{\Inn}{\mathrm{Inn}}

\newcommand{\PSL}{\mathrm{PSL}}

\newcommand{\C}{{\mathrm{C}}}

\newcommand\DD{{\mathscr{D}}}
\newcommand{\K}{\mathrm{K}}
\newcommand{\Diag}{\mathrm{Diag}}

\newcommand{\Ree}{\mathrm{Ree}}

\newcommand{\Sym}{\mathrm{Sym}}
\newcommand{\Alt}{\mathrm{Alt}}
\title[Highly-arc-transitive digraphs]{Quasiprimitive and bi-quasiprimitive highly-arc-transitive digraphs and finite simple groups}
\thanks{This work is supported by the Deutsch Forschungs-gemeinschaft (DFG), Germany Reseacrh Foundation)-Project-ID: 491392403-TRR 358}

\author{Lei Chen}
\address[Lei Chen]{Faculty of Mathematics, Bielefeld University, Bielefeld 33615}
\email{Lei.Chen@math.uni-bielefeld.de}

\author{Cheryl E. Praeger}
\address[Cheryl E. Praeger]{Department of Mathematics and Statistics, The University of Western Australia, Perth WA 6009}
\email{Cheryl.Praeger@uwa.edu.au}
\date{}
\tikzstyle{level 1}=[level distance=3.5cm, sibling distance=3.5cm]
\tikzstyle{level 2}=[level distance=3.5cm, sibling distance=2cm]
\begin{document}

\tikzstyle{bag} = [text width=4em, text centered]
\tikzstyle{end} = [circle, minimum width=3pt,fill, inner sep=0pt]

\begin{abstract}
    We extend the notion of an $H$-normal quotient digraph of an $H$-vertex-transitive digraph to that of an $H$-subnormal quotient digraph. Using these concepts, together with bipartite halves of bipartite digraphs, we show that, for each finite connected  $H$-vertex-transitive, $(H,s)$-arc-transitive digraph with $s\geqslant6$, either some $H$-normal quotient is a directed cycle of length at least $3$, or there is an $(L,t)$-arc-transitive digraph with $t\geqslant (s-3)/2$, and $L$ a vertex-quasiprimitive almost simple group with socle a composition factor of $H$.  This connection demonstrates that, to understand finite $s$-arc-transitive digraphs with large $s$, those admitting a vertex-quasiprimitive almost simple $s$-arc-transitive subgroup of automorphisms play a central role. We show that for each  $s$ and each odd valency $k$, there are infinitely many $(H,s)$-arc-transitive digraphs of valency $k$ with $H$ a finite  alternating group.
   
    In addition we discovered a novel construction which takes as input a connected non-bipartite $H$-vertex-transitive, $(H,s)$-arc-transitive digraph, and outputs a connected bipartite $G$-vertex-transitive, $(G,2s)$-arc-transitive digraph with $G=(H\times H).2$. This leads to construction of vertex-bi-quasiprimitive $s$-arc-transitive digraphs, for arbitrarily large $s$.   Our investigations yield several new open problems.



\end{abstract}

\maketitle

\section{Introduction}\label{s:intro}

The investigation of finite highly-arc-transitive digraphs traces back at least to 1989, when the second author~\cite[Theorem 2.8]{Praeger} proved that, for any positive integers \(s, k\), there exist  infinitely many finite connected \(s\)-arc-transitive digraphs of valency \(k\). She also proved~\cite[Lemma 3.2]{Praeger} that the \(s\)-arc-transitivity of a finite \(G\)-vertex-transitive digraph \(\Gamma\) is inherited by each $G$-normal quotient $\Gamma_N$ relative to a normal subgroup \(N\unlhd G\) with more than two orbits on vertices (Definition~\ref{d:nquot}). Our aim is to demonstrate that finite connected digraphs admitting a \emph{vertex-quasiprimitive almost simple $2$-arc-transitive} subgroup of automorphisms play a central role in describing the family of finite connected $s$-arc-transitive digraphs with $s\geq6$. 

In this paper, a digraph $\Gamma$ consists of a set $V\Gamma$ of vertices and a set $E\Gamma \subseteq V\Gamma\times V\Gamma$ of
arcs (or edges) such that  if $(u,v) \in E\Gamma$ then  $(v,u) \not\in E\Gamma$. In particular if $(u,v)$ is an arc then $u\ne v$, that is, a digraph has no loops. For a non-negative integer $s$, an $s$-arc in a digraph $\Gamma$ is a sequence $(v_0,\dots,v_s)$ of $s+1$ vertices such
that $(v_i,v_{i+1}) \in E\Gamma$ for $0 \leqslant i < s$. This means in particular that $v_{i-1}, v_i,v_{i+1}$ are pairwise distinct for $0< i < s$. The automorphism group $\Aut(\Gamma)$ consists of all permutations of $V\Gamma$ which leave the arc set $E\Gamma$ invariant, and for $G\leq \Aut(\Gamma)$, $\Gamma$ is said to be \emph{$(G,s)$-arc-transitive} if $G$ is transitive on the set of $s$-arcs. We sometimes refer to the \emph{underlying undirected graph} of a digraph $\Gamma$: this is the graph with vertex set  $V\Gamma$ and edges the unordered pairs $\{u,v\}$ such that $(u,v)\in E\Gamma$.  If $\Gamma$ is a finite connected digraph that is $(G,s)$-arc-transitive, for some $s\geq1$, and each vertex $u$ is the initial vertex of some arc $(u,v)$, then in particular $G$ is vertex-transitive. We say that $G$ is \emph{quasiprimitive} on $V\Gamma$ if each nontrivial normal subgroup of $G$ is transitive on $V\Gamma$, and that $G$ is \emph{bi-quasiprimitive} on $V\Gamma$ if each nontrivial normal subgroup of $G$ has at most two (equal length) orbits in $V\Gamma$, and some normal subgroup has two orbits in $V\Gamma$.  

The directed cycle $\C^{\to}_r$, for $r\geqslant3$, has automorphism group $Z_r$ and is $s$-arc-transitive for all $s$, and such graphs are regarded as rather degenerate examples.  
Each of the $(G,s)$-arc-transitive transitive digraphs constructed in \cite[Definition 2.6, Theorem 2,8]{Praeger} has a $G$-normal quotient that is a directed cycle $\C^{\to}_r$ of length $r$ for some $r\geq3$. On the other hand, a very different infinite family of examples was constructed in \cite[Theorem 1]{CLP}, namely for all integers $k\geq 2, s\geq 1$, and any integer $n>(s+1)k$ with $n$ coprime to $s+1$, a $(G,s)$-arc-transitive digraph of valency $k$ was explicitly constructed with $G$ the alternating group $\Alt(n)$ or symmetric group $\Sym(n)$ of degree $n$ (depending on the parities of $s, k$ and $n$), and $G$ acting regularly on the set of $s$-arcs. Clearly none of these digraphs has a cyclic $G$-normal quotient of order greater than $2$.  

Our aim is to study $G$-vertex-transitive $(G,s)$-arc-transitive digraphs with no cyclic $G$-normal quotients, that is, the following family of graphs/groups: 
\begin{equation}\label{e:family}
    \DD = \big\{ (\Gamma, G, s) \mid \begin{array}{l}
         \Gamma \ \mbox{finite connected digraph, } G\leq \Aut(\Gamma), s\geqslant 2; \Gamma\, \mbox{is $G$-vertex-transitive}  \\
          \mbox{and $(G,s)$-arc-transitive with no $G$-normal quotients $\C_r^\to$ for any $r\geqslant3$ }\\
    \end{array}
   \big\}
\end{equation}
We are particularly interested in members with large values of $s$, and we show (Theorem~\ref{t:main}) that each $(\Gamma, G, s)\in\DD$ with $s\geqslant6$ can be associated with a possibly smaller $(\Gamma^*, G^*, s^*)\in\DD$ where the group $G^*$ is an \emph{almost simple group}, that is, $T\unlhd G^*\leqslant \Aut(T)$ for some nonabelian simple group $T$, and $G^*$ acts quasiprimitively on $V\Gamma^*$ with $s^*$ at least  roughly $s/2$. Our analysis consists of several reduction steps of the types described in Definition~\ref{d:reln}. 

\begin{definition}\label{d:reln}
    {\rm
    For $ (\Gamma^{(1)},G^{(1)},s^{(1)})\in\DD$, we write  
    \[ 
    (\Gamma^{(1)},G^{(1)},s^{(1)}) \ \succ\   (\Gamma^{(2)},G^{(2)},s^{(2)})
    \]
if $\Gamma^{(2)}$ is a $G^{(1)}$-normal quotient of  $\Gamma^{(1)}$ (as in Definition~\ref{d:nquot} with $s^{(1)}\geq 2$), or $\Gamma^{(2)}$ is a  $G^{(1)}$-subnormal  quotient of  $\Gamma^{(1)}$ (as in Definition~\ref{d:nquot}  with $s^{(1)}\geq 3$),  or $\Gamma^{(2)}$ is a bipartite half of  $\Gamma^{(1)}$ (as in Definition~\ref{d:biphalf} with $s^{(1)}\geq 4$), and $(G^{(2)},s^{(2)})$ is as in Proposition~\ref{p:quot}(a) or (b), or Proposition~\ref{p:bip}(a), respectively. It follows from Propositions~\ref{p:quot} and~\ref{p:bip} that the triple $(\Gamma^{(2)},G^{(2)},s^{(2)})$ lies in $\DD$ in each case.
    }
\end{definition}

We can now state our main result.


\begin{theorem}\label{t:main}
Let $(\Gamma^{(1)},G^{(1)},s^{(1)})\in\DD$ with $s^{(1)}$ as in Table~\ref{tab:main}. Then there is a reduction sequence 
\[
(\Gamma^{(1)},G^{(1)},s^{(1)}) \ \succ\   (\Gamma^{(2)},G^{(2)},s^{(2)}) \ \succ\  \dots  \ \succ\  (\Gamma^{(n)},G^{(n)},s^{(n)})
\]
as in Definition~\ref{d:reln} with each $(\Gamma^{(i)},G^{(i)},s^{(i)})\in\DD$, such that $n, s^{(n)}$ are as in Table~\ref{tab:main}, and  $G^{(n)}$ is almost simple and vertex-quasiprimitive on $\Gamma^{(n)}$ with socle a composition factor of $G^{(1)}$. 
\end{theorem}

\begin{table}[]
    \centering
    \begin{tabular}{cccc}
    \hline
    $\Gamma^{(1)}$ bipartite  &$s^{(1)}\geqslant$  & possibilities for \(n\) &\(s^{(n)}\geqslant\) \\
       \hline
    no  & 3 & $n\in\{ 1, 2, 3\}$ & $s^{(1)}-1$ \\
    yes & 6 & $n\in\{2, 3,4,5\}$ & $\lfloor s^{(1)}/2\rfloor -1$ \\
    \hline
      \end{tabular}
    \caption{Parameter information for Theorem~\ref{t:main}}
    \label{tab:main}
\end{table}

The first step in the proof of Theorem~\ref{t:main} is to make a normal quotient reduction to a triple in $\DD$ with the group  quasiprimitive or bi-quasiprimitive on vertices and the integer `$s$' at least $3$ (Proposition~\ref{p:quot2}(c)). 
If the group is vertex-quasiprimitive we build on work of  Giudici and Xia~\cite{quasi} who proved that such a group must be  of type AS (almost simple) or PA. This analysis is given in Section~\ref{section:quasi}. 
On the other hand, if the group action is bi-quasiprimitive, then the original digraph, as well as the bi-quasiprimitive normal quotient digraph, is bipartite, and we reduce further to the bipartite half as in Proposition~\ref{p:bip}. A crucial tool we exploit is a detailed description of the group structure of finite bi-quasiprimitive groups by the second author in \cite{Praeger03}. Although this group theoretic information has been used in the study of finite \(2\)-arc-transitive bipartite graphs, it has not to our knowledge been used before to analyse digraphs (a very different situation). We make a delicate analysis of these actions on digraphs in Section~\ref{section:biquasi}. The formal proof of Theorem~\ref{t:main} is given in Section~\ref{proof of t:main}.

\subsection{Commentary and open problems}

We finish this section with some remarks and commentary, including some open problems that arise naturally from considering Theorem~\ref{t:main}. 

\subsubsection{The simple groups arising in Theorem~\ref{t:main}.}\label{s:simple}
The nonabelian simple group arising as the socle of the final group $G^{(n)}$ in Theorem~\ref{t:main} is a composition factor of the initial group $G^{(1)}$, and all the groups $G^{(i)}$ are sections of the initial group $G^{(1)}$. To see this, note that the group arising from a $G$-normal quotient reduction, or a $G$-bipartite-half reduction, is a quotient of $G$ or a normal subgroup of $G$, respectively (Proposition~\ref{p:quot}(a) and Proposition~\ref{p:bip}(a)). For a $G$-subnormal quotient reduction, the group arising is $N_G(N)/N$ with a nontrivial normal subgroup $M/N$ (Proposition~\ref{p:quot}(b)). Such a reduction is only used in the proof of Theorem~\ref{t:main} at the last step: the group $G$ is quasiprimitive of type PA with socle $M$, and $M/N$ is the nonabelian simple socle of the final almost simple group.    

Thus, for a finite connected $(G,s)$-arc-transitive digraph with large $s$, it follows from Theorem~\ref{t:main} that some nonabelian composition factor of $G$ is the socle of an almost simple group $G^*$ for which there exists a finite connected $(G^*,s^*)$-arc-transitive digraph with $s*\geqslant \lfloor s/2\rfloor -1$. 

\begin{problem}\label{p1}
    For a given finite nonabelian simple group $T$, decide if there exists a $(G,s)$-arc-transitive digraph with $G$ vertex-quasiprimitive, $T$ the socle of $G$, and $s\geqslant3$. Which families of simple groups arise with $s$ unboundedly large?
\end{problem}

Examples where $T$ is an alternating group $A_n$ with $n$ arbitrarily large occur as possibilities in Problem~\ref{p1}, see \cite[pp 66--67]{CLP}. In Section \ref{Sec:Ex} we explore the construction in \cite{CLP} and show that, for each even integer $s$ and each odd integer $k>1$, the construction produces  connected $(G,s)$-arc-transitive digraphs of valency $k$, such that infinitely many of these are $G$-vertex-quasiprimitive with $G$ a finite alternating group, and infinitely many of them are $G$-vertex-bi-quasiprimitive with $G$ a finite symmetric group, see Proposition~\ref{ex: alt and sym 2}.

In forthcoming work~\cite{LP} the authors show that large dimensional special linear groups also yield highly arc-transitive digraphs as in Problem~\ref{p1}.  

We note that $(G,s)$-arc-transitive, \emph{$G$-vertex-primitive} digraphs with $G$ almost simple exist for $s=2$. The first (and so far only) examples were given by Giudici, Li and Xia in \cite{infinite}, with 
\(G=\PSL_{3}(p)\) and \(p\equiv \pm2\mod{5}\), \(p\geqslant7\). Moreover  it was conjectured in \cite[Question 2]{infinite} that there is a bound on $s$ for such examples, and the first author conjectured in \cite[Conjecture 1.3]{LChenThesis} that this upper bound should be $2$, that is, the conjecture is that there do not exist any $(G,3)$-arc-transitive, $G$-vertex-primitive digraphs with $G$ almost simple. This conjecture has been proved for many families of simple groups, see \cite{alternating2, reeandsuzuki,symplectic, alternating, exceptional}. Thus we expect that the $(G,3)$-arc-transitive examples for Problem~\ref{p1} will have the group $G$, in its vertex action, being a \emph{`quimp group'}, that is, quasiprimitive and imprimitive. 

\subsubsection{Vertex-quasiprimitive $s$-arc-transitive graphs of PA type}

In \cite[Lemma 2.7]{quasi} and in \cite[Proof of Proposition 4.2]{Praeger} constructions were given for $(G,s)$-arc-transitive digraphs with $G$ vertex-quasiprimitive of PA type. In \cite{quasi}, these digraphs were called \emph{direct products} of the input digraphs. We give here a version relevant for us (which is not covered explicitly by either of the cases in \cite{quasi, Praeger}).

\begin{construction}\label{ex:PA}
{\rm 
    Let $s, m$ be integers such that $s\geq2$ and  $m\geq2$. Let $\Delta=(V_0,A_0)$ be a finite connected $(X,s)$-arc-transitive digraph. Define the \emph{direct product} $\Gamma=(V,A)$ to be the  digraph such that $V=V_0^m$ and $((u_1,\dots,u_m),(v_1,\dots,v_m))\in A$ if and only if $(u_i,v_i)\in A_0$ for each $i=1,\dots,m$. Then
    \begin{enumerate}
        \item[(a)]  $\Gamma$ is an $(X^m,s)$-arc-transitive digraph, by \cite[Lemma 2.7]{quasi};
        \item[(b)]  $\Aut(\Gamma)$ contains $X\wr \Sym(m)$. (This assertion follows from the proof of  \cite[Proposition 4.2]{Praeger}.)
        \item[(c)] If $X$  is almost simple and vertex-(quasi)primitive on $\Delta$, then the wreath product $X\wr \Sym(m)$ is  vertex-(quasi)prim\-it\-ive on $\Gamma$ of type PA (see for example, \cite[Theorem 5.18]{csaba}). 
        \item[(d)] In particular, if $\Delta$ is one of the $X$-vertex-primitive $(X,2)$-arc-transitive digraphs constructed in \cite[Theorem 1.1]{infinite}, with $X=\PSL_3(p^2)$ and $p\equiv \pm 2 \pmod{5}$, then $\Gamma$ is an $(X\wr \Sym(m),2)$-arc-transitive digraph with $X\wr \Sym(m)$ vertex-primitive of type PA.
    \end{enumerate}
    }
\end{construction}
Given the discussion in Subsection~\ref{s:simple}, we do not anticipate finding $s$-arc-transitive vertex-primitive examples of type PA with $s\ge3$. 
On the other hand taking $\Gamma_0$ to be one of the examples from \cite{CLP} yields examples of $(H\wr \Sym(m),s)$-arc-transitive digraphs with $H\wr \Sym(m)$ vertex-quasiprimitive of type PA (but not vertex-primitive), for and $m\geqslant2$ and arbitrarily large $s$ (Proposition~\ref{ex alt and sym 3}). 

A different aspect is the following: for a finite quasiprimitive permutation group $G$ of type PA on a set $\Omega$, the socle of $G$ is of the form $T^m$ for some nonabelian finite simple group $T$ and $G\leq H\wr \Sym(m)$, where $H\leq \Sym(V_0)$ is almost simple with socle $T$. However in general the set $\Omega$ is not equal to $V_0^m$ (which was the case  in Construction~\ref{ex:PA}); rather $G$ preserves a (not necessarily trivial) partition $\Sigma$ of $\Omega$, such that $\Sigma$ can be identified with  $V_0^m$. In \cite[Example 2.3]{quasi}, an explicit infinite family of  examples was given of $(G,2)$-arc-transitive digraphs with $G$ vertex-quasiprimitive of type PA where the parameter $m=2$ and the vertex set could not be identified with $V_0^2$. 

\begin{problem}\label{p2}
    Let $m, s$ be integers such that $\min\{m,s\}\geq2$ and $(m,s)\ne(2,2)$.  Decide whether or not there exist connected $G$-vertex-quasiprimitive, $(G,s)$-arc-transitive digraphs, with $G \leqslant \Sym(V_0)\wr \Sym(m)$ of type $PA$, and with $V\Gamma$ not equal to $V_0^m$.
\end{problem}

\subsubsection{ Vertex-quasiprimitive \(s\)-arc-transitive digraphs of diagonal type {\rm (SD and CD)}} \label{sub:SD}
The first main result of Giudici and Xia~\cite[Theorem 1.2]{quasi} gives a construction and classification of vertex-quasiprim\-it\-ive, \(2\)-arc-transitive digraphs of simple diagonal type SD, proving also that the examples they give are vertex-primitive, and that no $3$-arc-transitive digraphs of SD-type exist. We give more details of this construction in Remark~\ref{r:sd}. These examples of SD-type yield examples of compound diagonal type CD using Construction~\ref{ex:PA}, and it was shown in \cite[Corollary 1.4 ]{quasi} that all $2$-arc-transitive, vertex-quasiprimitive digraphs of type CD arise from this construction, and that none are $3$-arc-transitive. 

From this discussion, and in light of the reduction strategy we use requiring subnormal quotients (see Proposition~\ref{p:quot}), it seems unlikely that we could achieve a version of Theorem~\ref{t:main} for non-bipartite graphs with the parameters $s^{(1)}=s^{(n)}=2$. At least a different final step of the reduction pathway would be needed. So taking $s^{(1)}=3$ for non-bipartite digraphs seems the best we can do.


Also, we questioned whether our result for the bipartite case was best possible. For example, if \(4\leq s^{(1)}\leq5\), then in the reduction path we may potentially obtain a `bipartite half' that is  \(2\)-arc-transitive  and quasiprimitive of type SD or CD (see Lemma~\ref{l:i-or-iia}(a) and Proposition~\ref{Biquasi1}(b)).   It would be interesting to know if this situation can arise. We comment more on this  in Remark~\ref{r:sd}.

\begin{problem}\label{p4}
    Do there exist \(G\)-bi-quasiprimitive, \((G,s)\)-arc-transitive digraphs, where $s\in\{4,5\}$ and a bipartite half is a  \(2\)-arc-transitive digraph that is vertex-quasiprimitive of type SD or CD?
\end{problem}

\subsubsection{ Vertex-bi-quasiprimitive \(s\)-arc-transitive digraphs} \label{sub:biqp}

We tried, but failed, to construct a `bipartite doubling' of a given vertex-quasiprimitive \(s\)-arc-transitive digraph $\Delta$ that aimed to produce a vertex-bi-quasiprimitive \(2s\)-arc-transitive digraph with bipartite halves isomorphic to $\Delta$. Our efforts instead led to a new construction of a vertex-transitive \(2s\)-arc-transitive digraph from a given vertex-transitive \(s\)-arc-transitive digraph. This is described in Construction~\ref{con1}, and we explore properties of the construction in Section~\ref{s:conn}. 

First we note that, if we use as input to Construction~\ref{con1} one of the triples $(\Delta,X,s)\in \DD$ discussed in Subsection~\ref{s:simple} with $X$ a finite alternating or symmetric group, then for each given valency $k$ and each integer $s$, we obtain infinitely many triples $(\Gamma,G,2s)\in D$ with digraphs $\Gamma$ of valency $k$ and $G$-vertex-biquasiprimitive as in Lemma~\ref{l:i-or-iia}(b), see Propositions~\ref{ex: alt and sym 2} and~\ref{ex alt and sym 3}. 

Alternatively, if we take as input to  Construction~\ref{con1} one of the $(X,2)$-arc-transitive $X$-vertex-primitive digraphs $\Delta$ of type SD from  Giudici and Xia~\cite[Theorem 1.2]{quasi} discussed in Section~\ref{sub:SD}, then by Proposition~\ref{p:SD} the output digraph $\Gamma$, with automorphism group $G$, is a $(G,4)$-arc-transitive $G$-vertex-bi-quasiprimitive digraph which is not $(G,5)$-arc-transitive, and we have $(\Gamma,G,4)\in\DD$.
Moreover, if we were to apply the reduction strategy in the proof of Theorem~\ref{t:main} to the triple $(\Gamma^{(1)},G^{(1)},s^{(1)}) = (\Gamma,G,4)$ (see the description in `Claim 3' in Section~\ref{proof of t:main}) then: 
\begin{itemize}
    \item since $\Gamma$ is $G$-vertex-bi-quasiprimitive, the first reduction we would make is to form the bipartite half $\Gamma^{(2)}$ of $\Gamma$. By Proposition~\ref{p:SD}, $\Gamma^{(2)}=\Delta\times\Delta$, the direct product  admitting $G^+= X\times X$, and we would obtain $(\Gamma^{(2)},G^{(2)},s^{(2)}) = (\Delta\times\Delta,G^+,2)\in\DD$, with $G^+$ not quasiprimitive on $V\Gamma^{(2)}$. 
    \item The second reduction would be to form the $G^+$-normal quotient as in Proposition~\ref{Biquasi1}, yielding $(\Gamma^{(3)},G^{(3)},s^{(3)}) = (\Delta,X,2)\in\DD$, see Proposition~\ref{p:SD}.
\end{itemize}
There would be in fact no further  reductions available which would allow us to stay inside the set $\DD$.  If we were to form an $X$-subnormal quotient to obtain an almost simple group action,  we could not guarantee obtaining a $2$-arc-transitive action (Proposition~\ref{p:quot2}(b)). Thus the reduction procedure would end with the triple $(\Delta,X,2)\in\DD$. These examples show that, for bipartite digraphs, Theorem~\ref{t:main} would fail if we took $s^{(1)}=4$. We do not have examples of $5$-arc-transitive digraphs to demonstrate a similar failure if  $s^{(1)}=5$.

\begin{problem}\label{p5}
    Do there exist \(G\)-bi-quasiprimitive, \((G,5)\)-arc-transitive digraphs $\Gamma$ for which a reduction strategy as in Section~\ref{proof of t:main} applied to $(\Gamma^{(1)},G^{(1)},s^{(1)}) = (\Gamma,G,5)$ fails to yield  $(\Gamma^{(n)},G^{(n)},s^{(n)})\in\DD$ with $G^{(n)}$ almost simple and vertex-quasiprimitive? 
\end{problem}

\subsubsection{Basic $s$-arc-transitive digraphs of cycle type}
We end our commentary by considering the smallest family of $s$-arc-transitive digraphs which do not yield members of $\DD$. 
As mentioned above, for each of the digraphs $\Gamma=\C_r(v,t)$ (with $1\le t\le r-s$) in \cite[Definition 2.6 and Theorem 2.8]{Praeger}, the full automorphism group is $G=S_v\wr Z_r$ and $\Gamma$ is  $(G,s)$-arc-transitive. Moreover the $G$-normal quotient $\Gamma_N$, for $N=S_v^r$, is the directed cycle $\C_r^\to$. These are not the only $s$-arc-transitive digraphs with cyclic normal quotients and $s$ arbitrarily large. For example, \cite[Definition 2.10]{Praeger} gives a more general construction, and \cite[Remark 2.12]{Praeger} asks `how typical these examples are of $(G,s)$-arc-transitive digraphs' with a $G$-normal quotient $\C_r^\to$. 

Suppose we restrict our attention to $(G,s)$-arc-transitive digraphs $\Gamma$ which are not cycles, which possess at least one $G$-normal quotient $\C_r^\to$ with $r\ge3$, and for which all $G$-normal quotients  $\Gamma_N$ with $1\ne N\lhd G$ are directed cycles. We call such digraphs \emph{basic of cycle type}.  

\begin{problem}\label{p3}
    Describe the 
    basic $(G,s)$-arc-transitive digraphs of cycle type with a $G$-normal quotient $\C_r^\to$, for given $r, s\ge3$. 
\end{problem}

\section{Reduction steps for $(G,s)$-arc-transitive digraphs}\label{s:redn}

In order to prove Theorem~\ref{t:main} we use three kinds of reduction steps for finite connected $(G,s)$-arc-transitive digraphs $\Gamma$. The first two involve forming quotient graphs where we amalgamate vertices but essentially keep the same adjacency relation, while the third kind changes the nature of adjacency. These are discussed in  Subsections~\ref{s:nquot} and~\ref{s:bipartite}, respectively.  

We often use the following fact: \emph{if a connected vertex-transitive digraph is $(G,s)$-arc transitive, then it is also $(G,s')$-arc transitive for all $s'\leqslant s$ (since each $s'$-arc can be extended to an $s$-arc).}  

\subsection{Normal and subnormal quotients}\label{s:nquot}

Throughout this subsection let  \(\Gamma=(V,A)\) be a finite connected digraph with vertex set $V$ and arc set $A\subseteq V\times V$, and let $N$ be a subgroup of $\Aut(\Gamma)$ that is intransitive on $V$.
We introduce our most basic notion of a quotient $\Gamma_N$, and specialise to normal and subnormal quotients.

\begin{definition}\label{d:nquot}
{\rm 
    Let $V_N$ denote that set of $N$-orbits in $V$ and define the \emph{quotient} \(\Gamma_N\) to have vertex set $V_N$ and arc set $A_N\subseteq V_N\times V_N$ consisting of all pairs $(U,U')\in V_N\times V_N$ such that,  for some $u\in U, u'\in U'$, the pair $(u,u')\in A$.  

    In particular suppose that there exists a vertex-transitive subgroup $G\leq \Aut(\Gamma)$ containing $N$. If $N\unlhd G$, then $\Gamma_N$ is called the \emph{$G$-normal quotient of $\Gamma$ relative to $N$}. Alternatively, if $N \unlhd M\unlhd G$, $N$ is not normal in $G$, and $M$ is transitive on $V$, then $\Gamma_N$ is called a \emph{$G$-subnormal quotient of $\Gamma$ relative to $N$}.
    }
\end{definition}

In general a quotient $\Gamma_N$ need not be a digraph in our sense, since there may be an arc $(U,U')\in A_N$, such that $(U',U)$ is also an arc, for example, $\Gamma_N$ may have loops.  However we are concerned with $(G,s)$-are-transitive digraphs $\Gamma$, for some $s\geq3$, and in this case all $G$-normal and $G$-subnormal quotients with at least three vertices are digraphs.  


\begin{proposition}\label{p:quot}
   Let \(\Gamma=(V,A)\) be a finite connected $G$-vertex-transitive  \((G,s)\)-arc-transitive digraph with \(s\geqslant1\) and $G\leq\Aut(\Gamma)$. Let $N\leq G$ with $|V_N|\geq3$ and $\Gamma_N$ as in Definition~\ref{d:nquot}. Then $\Gamma_N$ is connected, and moreover the following hold.
   \begin{enumerate}
       \item[(a)] If $N\unlhd G$ and  \(s\geqslant2\), then $\Gamma_N$ is a digraph and is $G/N$-vertex-transitive and $(G/N,s)$-arc-transitive. 
       \item[(b)] If $N\unlhd M\unlhd G$ and  \(s\geqslant3\), with $N$ not normal in $G$ and $M$ vertex-transitive on $\Gamma$, then the $G$-subnormal quotient $\Gamma_N$ is a digraph and is $N_G(N)/N$-vertex-transitive and  $(N_G(N)/N,s-1)$-arc-transitive. 
   \end{enumerate} 
\end{proposition}

\begin{proof}
    The fact that $\Gamma_N$ is connected follows immediately from the connectedness of $\Gamma$. If $N\unlhd G$ and  \(s\geqslant2\) as in part (a), then by \cite[Lemma 3.2]{Praeger}, $\Gamma_N$ is a digraph and is $(G/N,s)$-arc-transitive. On the other hand if $N\unlhd M\unlhd G$ with $M$ vertex-transitive and  \(s\geqslant3\), as in part (b), then by \cite[Corollary 2.11]{quasi}, $\Gamma$ is $(M,s-1)$-arc-transitive; and hence (since $s-1\geq2$) by \cite[Lemma 3.2]{Praeger}, $\Gamma_N$ is a digraph and is $(M/N,s-1)$-arc-transitive, and hence also  
    $(N_G(N)/N,s-1)$-arc-transitive
    \end{proof}

Next we examine normal quotients which are in a sense minimal. Recall the definitions of quasiprimitive and bi-quasiprimitive for permutation groups, and the directed cycle graph $\C^{\to}_r$,  from Section~\ref{s:intro}.

\begin{proposition}\label{p:quot2}
   Let \(\Gamma=(V,A)\) be a finite connected $G$-vertex-transitive \((G,s)\)-arc-transitive digraph with \(s\geqslant2\) and $G\leq\Aut(\Gamma)$. Let $N\unlhd G$ with $|V_N|\geq3$. Then the following hold.
   \begin{enumerate}
       \item[(a)] If  $N\leqslant M\unlhd G$ with $|V_M|\geqslant3$, then $M/N\unlhd G/N$ and $\Gamma_M\cong (\Gamma_N)_{M/N}$.
       \item[(b)] If $r\geqslant3$ and no $G$-normal quotient of $\Gamma$ is isomorphic to $\C^{\to}_r$, then  no $G/N$-normal quotient of $\Gamma_N$ is isomorphic to $\C^{\to}_r$.
       \item[(c)] There exists $M\unlhd G$ such that $N\leqslant M$ and $|V_M|\geqslant3$, and the $G$-normal quotient $\Gamma_M$ is a $(G/M,s)$-arc-transitive digraph with $G/M$ quasiprimitive or bi-quasiprimitive on $V_M$. Moreover, if $(\Gamma, G, s)\in\DD$ then also $(\Gamma_M, G/M, s)\in\DD$.
   \end{enumerate}
\end{proposition}

\begin{proof}
    (a) Suppose that $N\leqslant M\unlhd G$ with $|V_M|\geqslant3$. Then $M/N\unlhd G/N$ and each $M$-orbit $u^M\in V_M$ is a union of all the $N$-orbits of the form $(u^N)^x=u^{Nx}$ for $x\in M$, and it follows that the map $\phi:V_M\to (V_N)_{M/N}$ sending $u^M$ to the $M/N$-orbit in $V_N$ containing $u^N$ is a bijection. Also by Definition~\ref{d:nquot}, $\phi$ induces a bijection from the arcs of $\Gamma_M$ to the arcs of $(\Gamma_N)_{M/N}$, giving a graph isomorphism $\Gamma_M\to  (\Gamma_N)_{M/N}$.

    (b) If $(\Gamma_N)_{\overline{M}}\cong \C^{\to}_r$ for some normal subgroup $\overline{M}=M/N$ of $G/N$, then $N\leqslant M\unlhd G$ and by part (a), the $G$-normal quotient $\Gamma_M$ is isomorphic to $\C^{\to}_r$. Thus part (b) is proved.

    (c) Choose a subgroup $M$ that is maximal by inclusion such that $N\leqslant M$ and $|V_M|\geqslant 3$. Such a subgroup exists because $N$ has these two properties and $G$ is finite. By Proposition~\ref{p:quot}(a), $\Gamma_M$ is a $(G/M,s)$-arc-transitive digraph. Further, each nontrivial normal subgroup of $G/M$ is of the form $L/M$ for some $L\unlhd G$ properly containing $M$. It follows from the maximality of $M$ that $|V_L|\leq 2$. Thus $L/M$ has at most two orbits in $V_M$ and so $G/M$ is either quasiprimitive or bi-quasiprimitive on $V_M$.  Finallly,  if $(\Gamma, G, s)\in\DD$ then  no $G$-normal quotient of $\Gamma$ is isomorphic to $\C^{\to}_r$ for any $r\geq3$, and by part (b),  no $G/M$-normal quotient of $\Gamma_M$ is isomorphic to $\C^{\to}_r$ for any $r\geq3$. Hence $(\Gamma_M, G/M, s)\in\DD$.
\end{proof}

\subsection{Bipartite halves}\label{s:bipartite}

Recall that a digraph is said to be bipartite if its underlying undirected graph is bipartite. For a bipartite digraph we now define its bipartite halves.

\begin{definition}\label{d:biphalf}
{\rm 
    Let  \(\Gamma=(V,A)\) be a finite bipartite digraph with vertex set $V$ and arc set $A\subseteq V\times V$, and let $\Delta, \Delta'$ be the parts of the bipartition, so that each arc $(u,v)\in A$ consists of one vertex from each of $\Delta, \Delta'$. The \emph{bipartite halves} of $\Gamma$ are the pairs $(\Delta, A_2)$ and $(\Delta', A_2')$, where 
    $A_2=\{(u,w)\mid u, w\in \Delta, \exists\ v\in\Delta'\ \mbox{such that}\ (u,v), (v,w)\in A\}$ and 
    $A_2'=\{(u,w)\mid u, w\in \Delta', \exists\ v\in\Delta\ \mbox{such that}\ (u,v), (v,w)\in A\}$. 
    }
\end{definition}

If the bipartite digraph $\Gamma$ is vertex-transitive, then each automorphism interchanging the parts $\Delta, \Delta'$ of the bipartition induces an isomorphism between the two bipartite halves in Definition~\ref{d:biphalf}. Thus without loss of generality we work with  $(\Delta, A_2)$.  With our applications in mind we present the following result for the case where $\Gamma$ is $4$-arc-transitive. We show that $(\Delta, A_2)$ is a digraph with a single exception. It would be interesting to know what other exceptions may arise in the case, say, of bipartite $2$-arc-transitive digraphs.


\begin{proposition}\label{p:bip}
   Let \(\Gamma=(V,A)\) be a finite connected $G$-vertex-transitive  \((G,s)\)-arc-transitive digraph with \(s\geqslant4\) and $G\leq\Aut(\Gamma)$. Suppose that $\Gamma$ is bipartite and that $\Delta, \Delta'$ and $\Gamma'=(\Delta, A_2)$ are as in Definition~\ref{d:biphalf}. Let $G^+$ be the (index $2$) subgroup of $G$ fixing $\Delta, \Delta'$ setwise. Then the following hold.
   \begin{enumerate}
     \item[(a)]    $\Gamma'$ is connected, $G^+$-vertex-transitive and  \((G^+,\lfloor s/2\rfloor)\)-arc-transitive, and either $\Gamma'$ is a digraph or $\Gamma=\C_4^\to, G=Z_4$ and $\Gamma'=\K_2$;
     \item[(b)] provided $\Gamma\ne \C_4^\to$, there exists $N\unlhd G^+$ such that the set $V_N$ of $N$-orbits in $\Delta$ satisfies $|V_N|\geqslant3$, and the $G^+$-normal quotient $\Gamma'_N$ is a $(G^+/N,\lfloor s/2\rfloor)$-arc-transitive digraph with $G^+/N$ quasiprimitive or bi-quasiprimitive on $V_N$;
     \item[(c)]  if $\Gamma$ has no $G$-normal quotient $\C_r^\to$ for any $r\geq3$, then $\Gamma'$ has no $G^+$-normal quotient $\C_r^\to$ for any $r\geq3$. Further, if    $(\Gamma, G, s)\in\DD$ then also $(\Gamma', G^+, \lfloor s/2\rfloor)\in\DD$.
   \end{enumerate}
\end{proposition}

\begin{proof}
    (a) The fact that $\Gamma'$ is connected follows from its definition and the fact that $\Gamma$ is connected.
    Let $s_0:=\lfloor s/2\rfloor$ and note that $s_0\geqslant2$ and $2s_0\leqslant s$. Let $(u_0, u_1,\dots, u_{s_0})$ be an $s_0$-arc of $\Gamma'$. Then, by the definition of $\Gamma'$, there are vertices $v_0, v_1, \dots v_{s_0-1}\in \Delta'$ such that $(u_0, v_0, u_1, v_1,\dots, v_{s_0-1}, u_{s_0})$ is a $2s_0$-arc of $\Gamma$. Then since $\Gamma$ is \((G,s)\)-arc-transitive and $2s_0\leqslant s$, it follows that $\Gamma'$ is \((G^+,s_0)\)-arc-transitive. 
    
    Suppose that $\Gamma'$ is not a digraph. Then there exist vertices $u,w\in \Delta$ such that $(u,w), (w,u)\in A_2$. Thus there are vertices $v,v'\in\Delta'$ such that $(u,v,w)$ and $(w,v',u)$ are $2$-arcs of $\Gamma$, and since $\Gamma$ is a digraph, the pair $(v,u)\not\in A$, so $v'\ne v$. Therefore $(u,v,w,v',u)$ is a $4$-arc of $\Gamma$. Since $\Gamma$  is \((G,4)\)-arc-transitive, for each $4$-arc $(u_0, u_1,\dots, u_{4})$ of $\Gamma$ we must have $u_0=u_4$. Since $\Gamma'$ is connected and since $G_{u_0}=G^+_{u_0}$ is transitive on the $4$-arcs of $\Gamma$ with first entry $u_0$, this implies that each vertex $w\in\Delta\setminus\{u_0\}$ occurs as $w=u_2$ in some $4$-arc $(u_0, u_1,\dots, u_{4})$ of $\Gamma$. Hence $\Gamma'$ is the complete directed graph $\K_n^\to$, where $n=|\Delta|$. However as each  $4$-arc $(u_0, u_1,\dots, u_{4})$ of $\Gamma$ has $u_0=u_4$, it follows that each $2$-arc $(u_0, u_2,u_{4})$ of $\Gamma'$ has $u_0=u_4$, and this implies that $n=2$. Hence $\Gamma'$ is the complete undirected graph $\K_2$, and since $\Gamma$ is a digraph we have $\Gamma=\C_4^\to$ and $G=Z_4$.

    (b) By part (a), if $\Gamma\ne \C_4^\to$, then by part (a), $\Gamma'$ is a digraph and is \((G^+,s_0)\)-arc-transitive with $s_0=\lfloor s/2\rfloor\geqslant2$ and $|\Delta|\geqslant3$. Now the assertions of part (b) follow by applying Proposition~\ref{p:quot2}(c) with $(\Gamma',G^+, 1, N, s_0)$ in place of $(\Gamma, G, N, M, s)$.
    
    (c) Suppose now that $\Gamma$ has no $G$-normal quotient $\C_r^\to$ for any $r\geq3$, and suppose to the contrary that $N\unlhd G^+$ is such that the $G^+$-normal quotient $\Gamma'_N=\C_r^\to$, for some $r\geq3$, with $G^+/N=Z_r$. If $r_0\geq3$ is any divisor of $r$ we may replace $N$ by a possibly larger normal subgroup $N_0$ so that $\Gamma'_{N_0}=\C_{r_0}^\to$. Hence we may without loss of generality assume that either $r=4$ or $r$ is an odd prime. If $N\unlhd G$, then $\Gamma_N$ is a $G$-normal quotient with $2r$ vertices. However, in this case if $(U, U', W)$ is a $2$-arc of $\Gamma_N$ with $U, W$ being $N$-orbits in $\Delta$ and $U'\subset \Delta'$, then $(U,W)$ is an arc of $\Gamma'_N=\C_r^\to$;  it follows that every 2-arc $(U,U', W')$ of $\Gamma_N$ must have $W'=W$. Hence $\Gamma_N$ has out-valency $1$, and so $\Gamma_N$ is the directed cycle $\C_{2r}^\to$, which is a contradiction.

   Hence $N$ is not normal in $G$. Since $G$ is vertex-transitive on $\Gamma$, we may choose an arc $(u,v)\in A$ with $u\in\Delta, v\in\Delta'$, and an element $g\in G$ such that $u^g=v$. Then $g$ must interchange $\Delta$ and $\Delta'$ and hence $G=\langle G^+, g\rangle$ and $g^2\in G^+$. Now $M:=N^g\unlhd G^+$, and $M\ne N$ as $N$ is not normal in $G$, and we note that $M^g=N^{g^2}=N$ since $g^2\in G^+$. Let $X:=N\cap M$, so $X\unlhd G$ and $G^+/X$ is isomorphic to a subgroup of $G^+/N\times G^+/M\cong Z_r\times Z_r$ of order strictly greater than $r$ (as $M\ne N$). We conclude that either (i) $G^+/X\cong Z_r^2$, or (ii) $r=4$ and $G^+/X\cong Z_4\times Z_2$.   Consider the $G$-normal quotient $\Gamma_X$, which by Proposition~\ref{p:quot} is a $(G/X,4)$-arc-transitive digraph, and by assumption is not a directed cycle. Thus $|V\Gamma_X|\geq 2r$ and $\Gamma_X$ has valency $d\geqslant2$, so the number of $4$-arcs is $|V\Gamma_X|\cdot d^3\geq 16r$. As $\Gamma_X$ is $(G/X,4)$-arc-transitive, $|G/X|$ is divisible by $|V\Gamma_X|\cdot d^3$, so in particular $|G/X|\geq 16r$. However in case (ii) $|G/X|=16$, so we are in case (i) with $2r^2=|G/X|\geq |V\Gamma_X|\cdot d^3.\geqslant 16 r$. Thus $r\geqslant 8$; hence $r$ is an odd prime, $G^+/X=Z_r\times Z_r$. This implies that the stabiliser $G_U/X=G^+_U/X$ of an $X$-orbit $U$ is isomorphic to a subgroup of $Z_r$ and so cannot be transitive on the $d^3$ $4$-arcs of $\Gamma_X$ with first entry $U$.  This proves the first assertion of part (c), and the final assertion follows from the definition of $\DD$ in \eqref{e:family}.
   \end{proof}

\section{Quasiprimitive digraphs}\label{section:quasi}

In this section, we examine the automorphism groups of finite vertex-quasiprimitive  $3$-arc-transitive  digraphs. We use the subdivision of finite quasiprimitive permutation groups introduced in \cite[Section 5]{BCC97} and conveniently described in \cite[Table 7.1]{csaba}. Each finite quasiprimitive group has one of the following types:  \(HA\), \(HS\), \(HC\), \(TW\), \(SD\), \(CD\), \(PA\) or \(AS\). The first result shows that the only possible types for a vertex-quasiprimitive $3$-arc-transitive group $G$ acting on a digraph (apart from a directed cycle) are the types \(PA\) and \(AS\) (sometimes called product action and almost simple).  

For type \(PA\), the group $G$ embeds in a wreath product $H\wr \Sym(m)$ for some $m\geq2$. Further, $G$ induces a transitive subgroup of $\Sym(m)$ and a stabiliser $G_1$ in this action on $\{1,\dots,m\}$ projects onto the group $H$, which is vertex-quasiprimitive  of type $AS$ on a set $\Delta$. The permutation group $H$ on $\Delta$ is called the \emph{component} of the PA-type group $G$.  More details of this $G$-action are given in \cite[Section 5]{BCC97} and will be used in the proof.

\begin{proposition}\label{p:pa}
 Assume that \(\Gamma\) is a finite connected \(G\)-vertex-quasiprimitive \((G,s)\)-arc-transitive digraph with \(s\geqslant2\), and that \(\Gamma\) is not a directed cycle. 
 \begin{enumerate}
     \item[(a)]  Then \(G\) is of type \(AS\), \(SD\), \(CD\) or \(PA\), and if $s\geqslant3$ then $G$ is of type \(PA\) or \(AS\). 
     \item[(b)] Moreover, if $s\geqslant3$ and \(G\) is of type \(PA\), such that \(G\leqslant H\wr\Sym(m)\) with \(H\) the almost simple component of $G$ and $m\geqslant2$, then for
     $M=G\cap H^m$, $N= G\cap (1\times H^{m-1})$, and $K=G\cap (H\times (H\wr\Sym(m-1))$, we have
     \begin{enumerate}
         \item[(i)] $N<M\unlhd G$, $N$ is not normal in $G$, $K=N_G(N)$, and $K/N\cong H$;  and 
         \item[(ii)] the \(G\)-subnormal-quotient \(\Gamma_{N}\)  is a connected,  $H$-vertex-quasiprimitive, \((H,s-1)\)-arc-transitive digraph, and \(H\) has type \(AS\) with socle a composition factor of $G$. 
     \end{enumerate}
 \end{enumerate}
\end{proposition}

\begin{remark}\label{r:pa}
{\rm 
    In Proposition~\ref{p:pa} the assumption $s\geqslant 3$ is necessary since, as discussed in Subsection~\ref{sub:SD}, there exist finite connected  \((G,2)\)-arc-transitive digraphs with $G$ vertex-quasiprimitive of SD-type, see \cite[Theorem 1.2]{quasi}.
    
    }
\end{remark}


\begin{proof}
(a) If \(G\) is of type \(HA\), \(HS\), \(HC\) or \(TW\), then \(G\) contains a normal subgroup acting regularly on the vertex set of \(\Gamma\), see \cite[Section 5]{BCC97}. Since \(\Gamma\) is \((G,s)\)-arc-transitive and \(s\geqslant2\), it follows from \cite[Theorem 3.1]{Praeger} that \(\Gamma\) is a directed cycle, which is a contradiction.
Thus \(G\) is of type \(AS\), \(SD\), \(CD\) or \(PA\).
Further, if $G$ were of type \(SD\) or \(CD\), then by \cite[Theorem 1.2, Corollary 1.4]{quasi} the integer $s$ must be $2$. This proves part (a).

(b) Now suppose that \(G\) is of type \(PA\) and $s\geqslant3$. Then (see \cite[End of Section 5]{BCC97}), 
\[
\Soc(G)=T^m\leqslant G \leqslant W=H\wr \Sym(m) \leqslant \Sym(\Delta)\wr \Sym(m),
\]
where $m\geqslant 2$ and $T$ is a nonabelian simple group, $H$ is a quasiprimitive permutation group on $\Delta$ of type $AS$ with non-regular socle $T$, and $G$ acts transitively by conjugation on the $m$ simple direct factors of $N$. Choose $\delta\in\Delta$ and set $R:=T_\delta$, so $1<R<T$ (as $T$ is not regular) and $|\Delta|=|T:R|>2$ (since $T$ is transitive on $\Delta$ and $T$ is  a nonabelian simple group). There is a $G$-invariant partition $\Sigma$ of $V\Gamma$ (where the parts of $\Sigma$ may, or may not, have size 1), such that, for some $\sigma \in \Sigma$, $N_\sigma = R^m$, and for a vertex $\alpha \in \sigma$, the stabiliser $N_\alpha$ is a subdirect subgroup of $R^m$. 
In particular we identify $\Sigma$ with the cartesian product $\Delta^m$ with the natural product action of $W$. 

To explain what we mean by $H$ being the `almost simple component of $G$', note that \(G\) projects to a transitive subgroup of \(\Sym(m)\) in the natural action on \(\{1,\ldots,m\}\), and we let \(W_{1}=H\times (H\wr\Sym(m-1))\) be the stabiliser in \(W\) of \(1\in\{1,\ldots,m\}\). Then \(K=G\cap W_{1}\) has index $m$ in $G$, and $H$ is chosen such that $H=\varphi(G_1)$ where \(\varphi:W_1\to H\) is the natural projection to the first direct factor. Note that $M=G\cap H^m$ is normal in $G$ and so is transitive on $V\Gamma$; also $N= G\cap (1\times H^{m-1})\lhd M$. On the other hand $N_G(N) = G\cap W_1=K$ so $N$ is not normal in $G$; in fact $N$ is the kernel of the restriction $\varphi|_{G_1}$ and $H=\varphi(G_1)\cong K/N$. This proves part (b)(i).

Note that for each $N$-orbit in $\Sigma=\Delta^m$ the  $m$-tuples $(\delta_1,\dots,\delta_m)$ in the orbit have a fixed first entry. Thus $N$ has at least $|\Delta|$ orbits in $\Sigma$, and hence $N$ has at least $|\Delta|$ orbits in $V\Gamma$, and as we noted above, $|\Delta|>2$. It now follows from Proposition~\ref{p:quot}(b) that the $G$-subnormal quotient $\Gamma_N$ is a connected digraph and is $N_G(N)/N$-vertex-transitive and  $(N_G(N)/N,s-1)$-arc-transitive. Thus part (b)(ii) is proved, recalling that $N_G(N)/N\cong H$ with socle $T$.
\end{proof}


\section{Bi-quasiprimitive digraphs}\label{section:biquasi}

Now we analyse the bi-quasiprimitive possibilities from Proposition~\ref{p:quot}. Here we have a $G$-vertex-biquasiprimitive, $(G,s)$-arc-transitive, subgroup of automorphisms of a finite connected digraph $\Gamma$ which is not a directed cycle. Thus there exists an index $2$ (normal) subgroup $G^+$ of $G$ with two orbits $\Delta, \Delta'$ in $V\Gamma$, noting that such a subgroup $G^+$ is not necessarily unique. The group $G$ preserves the partition $\{\Delta,\Delta'\}$, and we write the second $G^+$-orbit $\Delta'$ and a `standard interchanging permutation' $(1,2)$ as
\[
\Delta'=\{\delta'\mid \delta\in\Delta\}\quad \mbox{and \quad $(1,2):\delta\to \delta',\ \delta'\to \delta$\quad  for $\delta\in\Delta$.}
\]
With this labelling of $V\Gamma$, and replacing $G$ if necessary by a conjugate in $\Sym(V\Gamma)$, we therefore have $G\leqslant B\rtimes \Sym(2)=H\wr \Sym(2)$ with base group $B=H^2$ and `top group' $\Sym(2)=\langle (1,2)\rangle$, and we write  
\begin{equation}\label{pe:biqp-setup}
    \mbox{$B= H^2=X\times Y$ with $X=B\cap (H\times 1)$ and $Y=B\cap(1\times H)$.}
\end{equation}
As $G$ is vertex-transitive, $G$ contains an element $g$ interchanging $\Delta$ and $\Delta'$, and since $G^+$ projects onto $Y$, we may choose such an element of the form $g=(x,1)(1,2)$ for some $(x,1)\in X$. Thus $G\leqslant H\wr \Sym(2) = (X\times Y)\langle g\rangle$.  Our first result restricts the subgroup $G^+$. Here, for $\varphi\in\Aut(H)$, we write \(\Diag_{\varphi}(X\times Y) = \{(h, h^\varphi)\mid h\in H\}\) for the corresponding diagonal subgroup of $B$. Also for $h\in H$ we write $\iota_h$ for the inner automorphism $\iota_h:y\to h^{-1}yh$.

\begin{proposition}\label{Biquasi0}
    Suppose that \(\Gamma\) is a finite connected \(G\)-vertex-bi-quasiprimitive, \((G,s)\)-arc-trans\-itive digraph,  which is not a directed cycle, where \(s\geqslant2\). Let \(G^+\) be a subgroup of \(G\) of index \(2\) with two orbits \(\Delta\) and \(\Delta'\) in \(V\Gamma\), let \(H\) be the permutation group induced by \(G^+\) on \(\Delta\), and let $B, X, Y$ and $g=(x,1)(1,2)$ be as above.  Then \(G^+=\Diag_{\varphi}(H\times H)\) for some \(\varphi\in\Aut(H)\) such that $\varphi^2=\iota_x\in\Inn(H)$. Moreover part (a)(i) or (c) of \cite[Theorem 1.1]{Praeger03} holds.
    \end{proposition}

The conclusion of Proposition~\ref{Biquasi0} is rather technical and the two cases will be explored further in this section. We note that in \cite[Theorem 1.1(a)(i)]{Praeger03}, the group $H$ is quasiprimitive on $\Delta$, while in \cite[Theorem 1.1(c)]{Praeger03}, $H$ has an intransitive minimal normal subgroup $R$ such that $R^\varphi\ne R$, $M=R\times R^\varphi$ is transitive on $\Delta$, and $N=\Diag_\varphi(M\times M)$ is a minimal normal subgroup of $G$.

\begin{proof}
    Since \(\Gamma\) is connected, there must be an arc $(\delta, \tau)$ with $\delta\in\Delta$ and $\tau\in\Delta'$. (In fact each arc is either of this form or the reverse $(\tau,\delta)$, since $G$ is arc-transitive.) The image $\delta^g = \delta^{(x,1)(1,2)} = (\delta^x)^{(1,2)}\in \Delta'$, and since $G^+$ is transitive on $\Delta'$, there is some $(z,y)\in G^+$ such that $\tau=((\delta^x)^{(1,2)})^{(z,y)}$. Note that 
    \[
    \tau=   \delta^{g(z,y)}= ((\delta^x)^{(1,2)})^{(z,y)} = ((\delta^x)^{(y,z)})^{(1,2)} = 
    (\delta^{xy})^{(1,2)}.
    \]
    Since $g(z,y)\in G\leqslant \Aut(\Gamma)$, the image $(\delta,\tau)^{g(z,y)}=(\tau,\tau^{g(z,y)})$ is also an arc of $\Gamma$ and $\mu:= \tau^{g(z,y)}\in (\Delta')^{g(z,y)}=\Delta$. Thus $(\delta,\tau,\mu)$ is a $2$-arc of $\Gamma$, and since $\Gamma$ is a digraph, we conclude in particular that $\delta\ne\mu$. Note that 
    \begin{equation}\label{e:mu}
    \mu= \tau^{g(z,y)} = ((\delta^{xy})^{(1,2)})^{g(z,y)}
    = (\delta^{xy})^{(1,2)(x,1)(1,2)(z,y)} = (\delta^{xy})^{(1,x)(z,y)} =\delta^{xyz} \in \Delta\setminus\{\delta\}.
    \end{equation}
    Since \(\Gamma\) is \((G,2)\)-arc-transitive, it follows from \cite[Lemma 2.3]{quasi} that \(G_{\tau}=G_{\delta\tau}G_{\tau\mu}\).
    
    Suppose for a contradiction that \(G^+\) is not a diagonal subgroup of $X\times Y$. Then it follows from \cite[Theorem 1.1]{Praeger03} that the subgroups \(G^+\cap X\) and \(G^+\cap Y\) are non-trivial, and one of the parts (a)(ii), (a)(iii) or (b) of \cite[Theorem 1.1]{Praeger03} holds. In all these cases, \(G^+\cap X\) and \(G^+\cap Y\)  act transitively on \(\Delta\) and \(\Delta'\), respectively, and these actions are faithful.

\medskip\noindent
\emph{Claim. \(G^+\cap X\) and \(G^+\cap Y\)  act regularly on \(\Delta\) and \(\Delta'\), respectively.}

    \smallskip
    Note that \(G_{\tau}=G^{+}_{\tau}\). Let \(\pi_{1}\) and \(\pi_{2}\) be the natural projections on \(X\) and \(Y\), respectively. Then  \(\pi_{1}(G_{\tau})\cong G_{\tau}/(G_{\tau}\cap Y)\). Since \(G\cap Y=G^+\cap Y\) is transitive on \(\Delta'\) the index \(|G^+\cap Y|/|G_{\tau}\cap Y|=|\Delta'|\), and since \(\tau^{G^+}=\Delta'\) and \(\pi_{1}(G^+)=X\),  
    \[
    |\Delta'|=\frac{|G^+|}{|G_{\tau}|}=\frac{|G^+\cap Y|\cdot|\pi_{1}(G^+)|}{|G_{\tau}\cap Y|\cdot|{\pi_{1}}(G_{\tau})|}=|\Delta'|\cdot\frac{|\pi_{1}(G^+)|}{|{\pi_{1}}(G_{\tau})|}=|\Delta'|\cdot\frac{|X|}{|\pi_{1}(G_{\tau})|}.
    \]
    It follows that \(\pi_{1}(G_{\tau})=X\), and so we have  the factorisation $X=\pi_{1}(G_{\tau})=\pi_{1}(G_{\delta\tau})\pi_{1}(G_{\tau\mu})$. 
    
    Since \(G_{\delta}\cap X\unlhd\, G_{\delta}\), it follows that \(\pi_{1}(G_{\delta}\cap X)\unlhd\,\pi_{1}(G_{\delta})\), and since by definition $\pi_{1}(G_{\delta}\cap X) = G_{\delta}\cap X$, we have $G_{\delta}\cap X \unlhd\, \pi_{1}(G_{\delta})$. Thus $\pi_{1}(G_{\delta\tau})\leqslant \pi_{1}(G_{\delta})\leqslant N_X(G_\delta\cap X)$. The same argument with $\mu$ in place of $\delta$ gives $\pi_{1}(G_{\tau\mu})\leqslant \pi_{1}(G_{\mu})\leqslant N_X(G_\mu\cap X)$. Now by \eqref{e:mu}, the subgroups $G_\delta\cap X = (G^+\cap X)_\delta$ and $G_\mu\cap X = (G^+\cap X)_\mu$ are conjugate in $X$, and both are proper subgroups of $X$. Further, if \(G^+\cap X\)  is not regular on \(\Delta\), then these subgroups are nontrivial and are not normal in $X$. In this case the factorisation $X=\pi_{1}(G_{\delta\tau})\pi_{1}(G_{\tau\mu})$ contradicts \cite[Lemma 3.1]{symplectic}. Hence \(G^+\cap X\)  is regular on \(\Delta\), and therefore also \(G^+\cap Y\)  is regular on \(\Delta'\), proving the Claim.
    
    \medskip
    Now \(G^+\cap X\) fixes $\tau\in\Delta'$ and is transitive on $\Delta$. Since $\Gamma(\tau)\subseteq \Delta$ it follows that equality holds, that is, $\Gamma(\tau)= \Delta$. A similar argument shows that $\Gamma(\delta)= \Delta'$, and hence $\Gamma$ is the complete bipartite digraph $K_{m,m}$ with bipartition $\{\Delta, \Delta'\}$ and with edges directed from $\Delta$ to $\Delta'$, where $m=|\Delta|$. This is a contradiction since this digraph is not vertex-transitive, whereas we are assuming that $G$ is transitive on $V\Gamma$.

    This contradiction shows that \(G^+\) is a diagonal subgroup of $X\times Y$, and hence \(G^+=\Diag_{\varphi}(H\times H)\) for some \(\varphi\in\Aut(H)\), and hence part (a)(i) or (c) of \cite[Theorem 1.1]{Praeger03} holds. Since $g=(x,1)(1,2)$ normalises $G^+$, for each $(h,h^\varphi)\in G^+$, the conjugate $(h,h^\varphi)^g= (h^\varphi,h^x)$ also lies in $G^+$, and hence $h^x=h^{\varphi^2}$ for each $h\in H$. Thus $\varphi^2=\iota_x\in\Inn(H)$.
\end{proof}



 We collect together in Remark~\ref{r:Biquasi} further information about both the biquasiprimitive group $G$ and the `bipartite half digraph'  induced on $\Delta$ as in Definition~\ref{d:biphalf}, in the case where $\Gamma$ is $(G,4)$-arc-transitive.

\begin{remark}\label{r:Biquasi}
{\rm 
    Let $\Gamma, s, G, G^+, g, \varphi, \Delta, \Delta', H$ be as in Proposition~\ref{Biquasi0}, so in particular $\Gamma$ is not a directed cycle, and suppose that $s\geqslant4$. Let    
    \(\Gamma'=(\Delta, A_2)\) be the bipartite half digraph defined in Definition~\ref{d:biphalf}. It follows from Proposition~\ref{p:bip} that
    \begin{enumerate}
        \item[(a)]  \(\Gamma'\) is a connected \(H\)-vertex-transitive \((H,\lfloor s/2\rfloor)\)-arc-transitive digraph, and \(\Gamma'\) is not a directed cycle; 
        \item[(b)] and if    $(\Gamma, G, s)\in\DD$ then also $(\Gamma', H, \lfloor s/2\rfloor)\in\DD$.
    \end{enumerate}
\noindent
Further, by Proposition \ref{Biquasi0}  together with \cite[Theorem 1.1]{Praeger03},  \(G^+=\Diag_{\varphi}(H\times H)\) and one of the following holds:
\begin{enumerate}
    \item[(c)] Either \(H\) is quasiprimitive on $\Delta$ and $g$ does not centralise $G^+$ (\cite[Theorem 1.1 (a)(i)]{Praeger03});
    \item[(d)] or \(H\) has a  minimal normal subgroup \(R\) such that \(R\) is intransitive on $\Delta$, \(R^{\varphi}\ne R\), 
    \(M:=R\times R^{\varphi}\) is transitive on \(\Delta\), and \(N:=\Diag_{\varphi}(M\times M)\) is a minimal normal subgroup of $G$ (\cite[Theorem 1.1 (c)]{Praeger03}). 
    
        
\end{enumerate}
}
\end{remark}    

Next we obtain further information about the group $G$.

\begin{lemma}\label{l:i-or-iia}
Let $\Gamma, s, G, H$ etc be as in Proposition~\ref{Biquasi0} with $s\geqslant4$. Then no normal subgroup of $H$ acts regularly on $\Delta$, and $\Soc(G)$ is insoluble and is the unique minimal normal subgroup of $G$.
Moreover,
\begin{enumerate}
    \item[(a)] in the case Remark~\ref{r:Biquasi}(c),  either the quasiprimitive $H$-action on $\Delta$ has type AS or PA, or $s\in\{4,5\}$ and the action is of type SD or CD;
    \item[(b)] in the case Remark~\ref{r:Biquasi}(d), we have  $N=\Soc(G)\cong T^{2k}$, where $R=T^k$ for some finite nonabelian simple group $T$ and integer $k\geqslant1$.
\end{enumerate} 
\end{lemma}


\begin{proof}
    By Remark~\ref{r:Biquasi}(a), the bipartite half digraph $\Gamma'$ is a connected $H$-vertex-transitive \((H,\lfloor s/2\rfloor )\)-arc-transitive digraph, and is not a directed cycle. Since $\lfloor s/2\rfloor \geqslant2$, it follows from \cite[Theorem 3.1]{Praeger} that no normal subgroup of $H$ acts regularly on $\Delta$.

    Consider first the group in Remark~\ref{r:Biquasi}(c). Here $H$ is quasiprimitive on $V\Gamma'=\Delta$, and since $\lfloor s/2\rfloor \geqslant2$, it follows from Proposition~\ref{p:pa}(a) that the type of the $H$-action on $\Delta$ is as in part (a).  In particular $\Soc(H)$ is insoluble and is the unique minimal normal subgroup of $H$. By Remark~\ref{r:Biquasi}(c), $g$ does not centralise $G^+$, and so the socle $\Soc(G)\leqslant G^+$. Since $G^+=\Diag_\varphi(H\times H)$, by Proposition~\ref{Biquasi0}, we conclude that $\Soc(G)=\Diag_\varphi(\Soc(H)\times \Soc(H))$, so $\Soc(G)$ is insoluble and is the unique minimal normal subgroup of $G$. 



    Consider now the group in Remark~\ref{r:Biquasi}(d). It follows from  \cite[Theorem 1.1 (c)]{Praeger03} that $N=\Soc(G)$. Here $N\cong M=R\times R^\varphi$. If $R$ were abelian then $M$ would be a transitive abelian permutation group and hence would be regular (see \cite[Theorem 3.2]{csaba}), which is a contradiction. Hence $R$ is a nonabelian minimal normal subgroup of $H$ so $R=T^k$, for some nonabelian simple group and $k\geqslant1$, and $R$  is characteristically simple (see \cite[Lemma 3.14]{csaba}); and $g$ interchanges $\Diag_\varphi(R\times R)$ and  $\Diag_\varphi(R^\varphi\times R^\varphi)$, so that $N\cong T^{2k}$ is the unique minimal normal subgroup of $G$.      
    \end{proof}

We make some remarks about Lemma~~\ref{l:i-or-iia}(a) relevant to Problem~\ref{p4}.

\begin{remark}\label{r:sd}
{\rm

Regarding Lemma~\ref{l:i-or-iia}(a), in Section~\ref{Sec:Ex}  we give families of examples where the $H$-action on the bipartite half $\Gamma'$ is quasiprimitive of type $AS$ or $PA$. We do not know of any examples  where this $H$-action is $2$-arc-transitive and quasiprimitive of type $SD$ or $CD$ with  $s\in\{4,5\}$, see Problem~\ref{p4}.  We would be interested to know if there are any such examples. We note that Giudici and Xia show in \cite[Theorem 1.2]{quasi} that all $(X,2)$-arc transitive digraphs $\Delta$ with $X$ quasiprimitive on $\Delta$ of type SD arise from \cite[Construction 3.1]{quasi}. In that construction $\Soc(X)=T^k$ for an arbitrary nonabelian finite simple group $T$  with $k=|T|$, and the group $X = \Soc(X)\rtimes T$ (see the line preceding \cite[Lemma 3.6]{quasi}). The group $X$ induces a primitive subgroup $T\times T$ of the holomorph ${\rm Hol}(T) < \Sym(k)$ on the simple direct factors of $\Soc(X)$, and in particular $X$ is vertex-primitive of type $SD$. Now it follows from \cite[Theorem, Section 3.4]{P90} that the  overgroups of $T\times T$ in $\Sym(k)$ are either subgroups of ${\rm Hol}(T).2$ or contain $\Alt(k)$. Giudici and Xia prove in \cite[Theorem 3.15]{quasi} that $\Aut(\Delta)$ induces precisely the subgroup ${\rm Hol}(T)$ of $S_k$ on the simple direct factors of $\Soc(X)$, and in fact that $\Aut(\Delta) = \Soc(X). \Aut(T)$, is vertex-primitive of type $SD$, and is $2$-arc-transitive but not $3$-arc-transitive on $\Delta$.    

  }  
\end{remark}

It remains to analyse further the case of Lemma~\ref{l:i-or-iia}(b) where \(H\) is not quasiprimitive on \(\Delta\). We note that, for each $s\geq4$ there are infinitely many possibilities for $\Gamma, s, G, H$ as in Proposition~\ref{Biquasi0}  such that Lemma~\ref{l:i-or-iia}(b) holds for $G$, see Proposition~\ref{ex alt and sym 3}.

\begin{proposition}\label{Biquasi1}
    Let $\Gamma, s, G, H$ etc be as in Proposition~\ref{Biquasi0} with $s\geqslant4$, and suppose that $H, R, M$, $N, \Gamma'$ are as in Remark~\ref{r:Biquasi}(d). Let $C:=C_H(R)$, $\overline{H}=H/C$, and $\overline{R}=RC/C$. Then
    \begin{enumerate}
        \item[(a)] the $H$-normal quotient $(\Gamma')_C$ is a connected  $\overline{H}$-vertex-quasiprimitive $(\overline{H}, \lfloor s/2\rfloor)$-arc-transit\-ive digraph that is not a directed cycle;
        \item[(b)] either the quasiprimitive $\overline{H}$-action on $V(\Gamma')_C$ has type AS or PA, or $s\in\{4,5\}$ and the action has type SD or CD. 
    \end{enumerate}    
\end{proposition}

In the next section we introduce a new digraph construction that can be applied to produce a $(G,4)$-arc-transitive $G$-vertex-bi-quasiprimitive digraph $\Gamma$ such that  Lemma~\ref{l:i-or-iia}(b) holds for $H$, and the normal quotient operation in Propostition~\ref{Biquasi1} produces a digraph admitting a vertex-primitive $\overline{H}$-action os type \(SD\), see Proposition~\ref{p:SD}.

\begin{proof}
    By Lemma~\ref{l:i-or-iia}, $R=T^k$, for some finite nonabelian simple group $T$ and integer $k\geq1$.  Now $C\unlhd H$ (as $R\unlhd H$), and $C\cap R=1$ (as $R$ has trivial centre), and $M\cap C=R^\varphi$ (as $M=R\times R^\varphi$). Recall that $R$ is intransitive on $\Delta$ and $R^\varphi$ acts transitively on the $R$-orbits in $\Delta$ (since $M$ is transitive on $\Delta$). Hence $R$, and similarly also $R^\varphi$, have more than two orbits in $\Delta$. 

\smallskip\noindent
\emph{Claim 1: $C$ is intransitive, with more than $2$ orbits,  on $\Delta$, and $R$ acts transitively on the $C$-orbits.}


Suppose to the contrary that $C$ is transitive on $\Delta$. Consider the $H$-normal quotient $(\Gamma')_{R^\varphi}$. Then since $\Gamma'$ is a connected $H$-vertex-transitive, $(H,\lfloor s/2\rfloor)$-arc-transitive digraph that is not a directed cycle (by Remark~\ref{r:Biquasi}(a)), it follows from Proposition~\ref{p:quot}(a) that $(\Gamma')_{R^\varphi}$ is a connected $(H/R^\varphi,\lfloor s/2\rfloor)$-arc-transitive digraph. Further, both $M/R^\varphi\cong R= T^k$ and $C/R^\varphi$ are vertex-transitive normal subgroups of $H/R^\varphi$. Since $C=C_H(R)$ and $M=R\times R^\varphi$, it follows that $C/R^\varphi$ centralises $M/R^\varphi$, and this implies, by \cite[Theorem 3.2]{csaba}, that $M/R^\varphi$ acts regularly on the  vertices of $(\Gamma')_{R^\varphi}$. Since  $M/R^\varphi$ is insoluble, the quotient   $(\Gamma')_{R^\varphi}$ is not a directed cycle, and we have a contradiction to \cite[Theorem 3.1]{Praeger} (that for such digraphs $H/R^\varphi$ has no vertex-regular normal subgroup). Hence $C$ is intransitive on $\Delta$. Finally  since $CR$ contains $M$, and $M$ is transitive on $\Delta$, it follows that $R$ acts transitively on the $C$-orbits in $\Delta$, and as $R$ has no subgroup of index $2$, $C$ must have more than two orbits in $\Delta$, proving  Claim 1.

\smallskip\noindent
\emph{Claim 2: The quotient $\overline{R}\cong R=T^k$, and $\overline{R}$ is the unique minimal normal subgroup of $\overline{H}$.}

The conjugation action of $H$ on $R$ induces a group homomorphism $H\to \Aut(R)$ with kernel $C = C_H(R)$, and hence we may identify $\overline{H}=H/C$ with a subgroup of $\Aut(R)=\Aut(T)\wr S_k$. The group $\overline{R}$ is then identified with the group of inner automorphisms ${\rm Inn}(R)$, and $\overline{R} = RC/C\cong R/(R\cap C)\cong R= T^k$. Also $\overline{R}$ is a minimal normal subgroup of $\overline{H}$ since $R$ is a minimal normal subgroup of $H$. Finally a second minimal normal subgroup of $\overline{H}$ would necessarily intersect $\overline{R}$ trivially and so would centralise $\overline{R}$. However $\overline{R}$ has trivial centraliser in $\Aut(R)$, and we conclude that $\overline{R}$ is the unique minimal normal subgroup of $\overline{H}$, and Claim 2 is proved.
   
\smallskip
In the light of Claim 1 we can apply Proposition~\ref{p:quot}(a) to the $H$-normal quotient $(\Gamma')_{C}$ and conclude that $(\Gamma')_{C}$ is a connected $\overline{H}$-vertex-transitive, $(\overline{H},\lfloor s/2\rfloor)$-arc-transitive digraph. By Claims 1 and 2, $\overline{H}$ has a unique minimal normal subgroup, namely $\overline{R}$, and $\overline{R}$ is vertex-transitive. Thus  $\overline{H}$ is vertex-quasiprimitive, and since $\overline{R}$ is insoluble, the digraph  $(\Gamma')_{C}$ is not a directed cycle. Thus part (a) is proved. 
Part (b) follows from Proposition~\ref{p:pa}(a).
  \end{proof}

\section{Construction of bi-quasiprimitive highly-arc-transitive digraphs}\label{s:conn}

In this section, we present a construction that, for a given  $s$-arc transitive digraph of order $n$ and valency $k\geq 2$ as input, produces a bipartite $2s$-arc-transitive digraph of order $2n^2$ with the same valency $k$. In order to obtain a connected bipartite digraph, the input digraph must be both connected and non-bipartite (Proposition~\ref{thm: bi-quasi con}). We have not been able to decide, for a given input triple $(\Delta, X, s)\in\mathcal{D}$, whether or not the output  triple $(\Gamma, G, 2s)$ of Construction~\ref{con1} necessarily lies in $\mathcal{D}$. It would be interesting to resolve this. 

\begin{problem}\label{p6}
    If $(\Delta, X, s)\in\mathcal{D}$ (as in \eqref{e:family}) with $\Delta$ connected and non-bipartite, decide whether or not the output triple $(\Gamma, G, 2s)$ from Construction~\ref{con1} must necessarily lie in $\mathcal{D}$.  
\end{problem}
Note that the construction is not the reverse of the bipartite halving construction in Definition~\ref{d:biphalf}, despite an apparent similarity in that both constructions would lead to doubling the arc-transitivity. Instead, the bipartite halves of the digraph $\Gamma$ in Construction~\ref{con1} are direct products of two copies of the input digraph $\Delta$, as defined in Construction~\ref{ex:PA}. 

{\color{purple}

}

\begin{construction}\label{con1}
{\rm
    Let $\Delta$ be a finite digraph of valency $k\geqslant2$, 
    and suppose that $X\leq \Aut(\Delta)$ is such that $\Delta$ is $X$-vertex-transitive and $(X,s)$-arc-transitive for some $s\geqslant2$. 
    \begin{enumerate}
    \item[(a)]   Define $\Gamma$ to be the bipartite digraph with vertex set 
        \[
    V\Gamma = \{ (v,w,\delta) \mid v,w\in V\Delta, \delta\in \mathbb{Z}_2=\{0,1\}\} = \Gamma_0\,\dot\cup\,\Gamma_1
    \]
    where  \(\Gamma_\delta:=\{(v,w,\delta):v,w\in V\Delta\}\), for \(\delta\in\{0,1\}\), are the parts of the bipartition, 
    and define the arc set $A\Gamma$ by
    \[
    \big((v,w,\delta),(v',w',\delta')\big)\in A\Gamma\quad \Longleftrightarrow\quad 
    \begin{cases}
       v'=v \ \mbox{and} \ (w,w')\in A\Delta  &\mbox{if $\delta=0$, $\delta'=1$}\\
       (v,v')\in A\Delta \ \mbox{and} \  w=w'  &\mbox{if $\delta=1$, $\delta'=0$.}\\  
    \end{cases} 
    \]
    \item[(b)] Let $G^+:=X\times X$, acting on $V\Gamma$ by $(x,y):(v,w,\delta)\to (v^x,w^y,\delta)$, for all $(v,w,\delta)\in V\Gamma, x,y\in X$.  
    
     \item[(c)] Let $(v_0,v_1)$ be an arc  of $\Delta$ and $g\in X$  such that $v_1=v_0^g$. 
   Let $t := (g,1)\tau$, where $\tau:(v,w,\delta)\to (w,v,\delta+1)$ for all $(v,w,\delta)\in V\Gamma$, and let $G=\langle G^+, t\rangle$. 
    \end{enumerate}
}
\end{construction}

For $v\in V\Delta$, we denote the set of out-neighbours of $v$ by $\Delta(v)=\{ v'|(v,v')\in A\Delta\}$ and the set of in-neighbours by  $\Delta^*(v)=\{ v'|(v',v)\in A\Delta\}$. Since $\Delta$ is finite, and $X$ is transitive on $V\Gamma$ and $A\Gamma$, we see that $|A\Delta|=|V\Delta|\cdot |\Delta(v)| = |V\Delta|\cdot |\Delta^*(v)|$, so that $\Delta$ has in-valency and out-valency  equal to $k$. Note that for infinite arc-transitive, vertex-transitive digraphs the out-valency and in-valency are not necessarily equal. We establish some of the basic properties of the digraph $\Gamma$.


\begin{proposition}\label{thm: bi-quasi con}
    Let $\Gamma, G$ be as in Construction~\ref{con1}. Then: 
        \begin{enumerate}
            
        \item[(a)] \(\Gamma\) is a bipartite digraph of valency \(k\), and \(\Gamma\) is connected if and only if \(\Delta\) is connected and not bipartite. Moreover each bipartite half of $\Gamma$ (as in Definition~\ref{d:biphalf}) is the direct product of two copies of $\Delta$ (as in Construction~\ref{ex:PA}).
        
        \item[(b)] 
        $G\leqslant \Aut(\Gamma)$, and acts transitively on \(V\Gamma\). Moreover $|G:G^+|=2$ with $t^2=(g,g)\in G^+$, and the $G^+$-orbits are $\Gamma_0$ and $\Gamma_1$, and are interchanged by $t$. 
        
        \item[(c)] For $r\geq1$, each $2r$-arc $(u_0,\dots, u_{2r})$ of $\Gamma$ with $u_0\in \Gamma_0$  satisfies
        \begin{align*}
            u_{2i} &= (v_i,w_i,0) &\mbox{for $0\leq i\leq r$}\\
            u_{2i+1} &= (v_i,w_{i+1},1) &\mbox{for $0\leq i\leq r-1$}
        \end{align*}
        such that $(v_0,v_1,\dots,v_r)$ and $(w_0,w_1,\dots,w_r)$ are $r$-arcs of $\Delta$.
        \item[(d)] The digraph $\Gamma$ is $(G,2s)$-arc-transitive.
    \end{enumerate}
\end{proposition}

\begin{proof}
(a)  It follows from the definition of $A
\Gamma$ in Construction~\ref{con1}(a) that $\Gamma$ is a digraph and is bipartite, and in view of our comment above, each vertex of $\Gamma$ has  out-valency and in-valency equal to $|\Delta(v)|=|\Delta^*(v)|=k$. Next we note that, for $\delta\in\{0,1\}$,  the bipartite half $(\Gamma_\delta,A_\delta)$, as defined in Definition~\ref{d:biphalf}, has arc set $A_\delta=\{((v,w,\delta), (v',w',\delta))| (v,v'), (w,w')\in A\Delta\}$, and hence is isomorphic to the digraph defined in Construction~\ref{ex:PA} (with $m=2$). From the discussion in Construction~\ref{ex:PA}, $G^+$ acts vertex-transitively and $s$-arc transitively on $(\Gamma_\delta,A_\delta)$. Moreover, the underlying undirected graph of $(\Gamma_\delta,A_\delta)$ is connected if and only if the digraph $\Delta$ is connected and not bipartite (see \cite[Section 6.3]{GR}). Therefore also the digraph $(\Gamma_\delta,A_\delta)$ is connected if and only if $\Delta$ is connected and not bipartite, see \cite[Theorem 1.10]{Cameron}. If this is the case, then for every pair of vertices in $\Gamma_0$ there is a directed path in $\Gamma$ of even length from the first to the second, and it follows from the definition of $A\Gamma$ that also $\Gamma$ is connected.   Thus part (a) is proved.

\smallskip
(b)     We first show that \(G^+\) preserves \(A\Gamma\). Let \((x,y)\in G^+\) and \(((v,w,\delta),(v',w',\delta'))\in A\Gamma\), so $(\delta,\delta')=(0,1)$ or $(1,0)$.  By the definition of the action of $(x,y)$, we have
    \[
    ((v,w,\delta),(v',w',\delta')^{(x,y)} = 
        ((v^x,w^y,\delta),(v'^x,w'^y,\delta')).   
    \]
   If $(\delta,\delta')=(0,1)$  then \(v=v'\) and \((w,w')\in A\Delta\), and so also  \(v^x=v'^x\) and \((w^y,w'^y)\in A\Delta\) since $y\in\Aut(\Delta)$; hence $((v,w,\delta),(v',w',\delta')^{(x,y)}\in A\Gamma$. Similarly, if $(\delta,\delta')=(1,0)$, then  \(w=w'\) and \((v,v')\in A\Delta\), and also  \(w^y=w'^y\) and \((v^x,v'^x)\in A\Delta\), since $x\in\Aut(\Delta)$; hence $((v,w,\delta),(v',w',\delta')^{(x,y)}\in A\Gamma$.  Thus \(G^+\) preserves \(A\Gamma\) and \(G^+\leqslant \Aut(\Gamma)\). 
    Since $X$ is transitive on $V\Delta$ it follows that the orbits of $G^+=X\times X$ in $V\Gamma$ are $\Gamma_0$ and $\Gamma_1$.

   Now  \(t: (v,w,\delta)\to (w,v^g,\delta+1)\) and consequently \(t^2:(v,w,\delta)\to (v^g,w^g,\delta)\), for each vertex $(v,w,\delta)$. Hence $t^2=(g,g)\in G^+$, and $t$ interchanges $\Gamma_0$ and $\Gamma_1$ which implies that $G$ is transitive on $V\Gamma$. 
   Also $(g,1)\in G^+$ and  $(x,y)^\tau =(y,x)$ for all $(x,y)\in G^+$, and therefore $G=\langle G^+, t\rangle = G^+\cdot \langle \tau\rangle$ and $|G:G^+|=2$.  It remains to show that  $t\in\Aut(\Gamma)$. Consider an arc  $a:= \big((v,w,\delta),(v',w',\delta')\big)$ of $\Gamma$, and note that $\delta\ne\delta'$. 
            Suppose first that $\delta=0$. Then $a:= \big((v,w,0),(v, w',1)\big)$ and $(w,w')\in A\Delta$. Now $a^t = \big((w,v^g,1),(w', v^g,0)\big)$ by the definition of $t$, and this is an arc by the definition of $A\Gamma$.
            Now suppose that $\delta=1$. Then $a:= \big((v,w,1),(v', w,0)\big)$ and $(v,v')\in A\Delta$. We have $a^t = \big((w,v^g,0),(w, v'^g,1)\big)$ by the definition of $t$, and, noting that $(v^g,v'^g)\in A\Delta$ since $g\in\Aut(\Delta)$, we conclude again that $a^t$ is an arc of $\Gamma$ by the definition of $A\Gamma$. Thus $t\in\Aut(\Gamma)$ and part (b) is proved.

\smallskip
(c) We prove part (c) by induction on $r$. If $r=1$ then part (c) follows directly from the definition of $A\Gamma$. So suppose that $r\geq2$, that   $(u_0,\dots, u_{2r})$ is a $2r$-arc of $\Gamma$ with $u_0\in \Gamma_0$, and suppose inductively that $u_0,\dots, u_{2r-2}$ are as in part (c) with  $(v_0,v_1,\dots,v_{r-1})$ and $(w_0,w_1,\dots,w_{r-1})$ being $(r-1)$-arcs of $\Delta$. Then by  the definition of $A\Gamma$, since $u_{2r-2}=(v_{r-1}, w_{r-1}, 0)$, we have 
$u_{2r-1}=(v_{r-1}, w_{r}, 1)$ and $u_{2r}=(v_{r}, w_{r}, 0)$ with $(w_{r-1},w_r), (v_{r-1},v_r)\in A\Delta$. Hence $(v_0,v_1,\dots,v_{r})$ and $(w_0,w_1,\dots,w_{r})$ are $r$-arcs of $\Delta$, and part (c) is proved by induction.

\smallskip
(d) By part (b), $\Gamma$ is $G$-vertex-transitive, and hence to prove that  $\Gamma$ is $(G, 2s)$-arc-transitive it is sufficient to prove that $G^+$ is transitive on the set of $2s$-arcs $(u_0,\dots, u_{2s})$ of $\Gamma$  with $u_0\in \Gamma_0$. Let  $(u_0',\dots, u_{2s}')$ be another $2s$-arc of $\Gamma$ with $u_0'\in \Gamma_0$.  By part (c) we may assume that  $(u_0,\dots, u_{2s})$ is as given there, and that each $u_{2i}'=(v_{i}', w_{i}', 0)$ and $u_{2i+1}'=(v_{i}', w_{i+1}', 1)$ with  $(v_0',v_1',\dots,v_{s}')$ and $(w_0',w_1',\dots,w_{s}')$ being $s$-arcs of $\Delta$. Since $\Delta$ is $(X,s)$-arc-transitive, there exist $x,y\in X$ such that  $(v_0,v_1,\dots,v_{s})^x=(v_0',v_1',\dots,v_{s}')$ and $(w_0,w_1,\dots,w_{s})^y=(w_0',w_1',\dots,w_{s}')$. Then $(x,y)\in G^+$ and $(u_0,\dots, u_{2s})^{(x,y)}= (u_0',\dots, u_{2s}')$, proving part (d). 
\end{proof}

{\color{blue}
}

Next we find sufficient conditions for Construction~\ref{con1} to produce bi-quasiprimitive digraphs.

\begin{proposition}\label{cor:biquasi}
    Let $\Delta, X, s, \Gamma, G$ be as in Construction~\ref{con1}, 
   suppose that \(\Delta\) is connected and non-bipartite, and that \(X\) has a unique minimal normal subgroup \(N\) and \(N\) is nonabelian. Then 
   \begin{enumerate}
       \item[(a)] $G$ has a unique minimal normal subgroup $L$ and \(L\cong N\times N\). 
       \item[(b)] If also  $\Delta$ is an $X$-vertex-quasiprimitive $(X,s)$-arc-transitive digraph, then \(\Gamma\) is a $G$-vertex-bi-quasiprimitive $(G,2s)$-arc-transitive digraph. 
   \end{enumerate} 
\end{proposition}

\begin{proof}
    Now $G^+ = X_1\times X_2$ with $X_1\cong X_2\cong X$, so setting $N_i=\Soc(X_i)\cong N$, for each $i$,  \(L:= N_1 \times N_2\cong N\times N\) and $L$ is characteristic in \(G^+\). This together with the fact that \(G^+\) is normal in \(G\) implies that \(L\) is also normal in \(G\). Let \(M\) be an arbitrary non-trivial normal subgroup of \(G\). 
    
    \medskip\noindent
    \textit{Claim:}\quad \(M\cap G^+\neq 1\) and \(M\cap X_1\neq 1\).
    
    \smallskip
    Suppose to the contrary that \(M\cap G^+=1\). Then $\langle M, G^+\rangle \cong M\times G^+$ and since \(M\neq 1\) and \(|G:G^+|=2\), it follows that $|M|=2$ and $G=M\times G^+$. Thus  \(M=\langle (a,b)\tau\rangle\cong Z_2\), for some \(a,b\in X\), and each nontrivial $(x,1)\in X_1$ commutes with $(a,b)\tau$, so 
    \[
    (xa, b)\tau = (x,1)(a,b)\tau = (a,b)\tau (x,1) = (a,b)(1,x)\tau = (a,bx)\tau
    \]
    which implies that $x=1$, a contradiction. Therefore \(M\cap G^+\neq 1\), so $M\cap G^+$ contains some \((a,b)\neq (1,1)\). Since also \((1,1)\neq (a,b)^{(g,1)\tau}=(b,a^g)\in M\cap G^+\), we may assume without loss of generality that \(a\neq1\). 
    Since \(N\) is the unique minimal normal subgroup of \(X\) and \(N\) is nonabelian, it follows that $X$ has trivial centre. Hence there exists \(c\in X\) such that \([c,a^{-1}]\neq 1\). Further, since \(M\cap G^+ \unlhd G^+\), also \((a,b)^{(c,1)}\in M\cap G^+\) and therefore
    \[
    (c^{-1},1)(a,b)(c,1)(a^{-1},b^{-1})=([c,a^{-1}],1)\in M\cap G^+
    \]
    and in fact \(([c,a^{-1}],1)\in M\cap X_1\neq 1\),  and the claim is proved.

\smallskip
    
      Since \(N_1\) is the unique minimal normal subgroup of \(X_1\) and \(M\cap X_1\neq 1\), it follows that \(M\cap X_1\geqslant N_1\). Further, since \(M\) is normalised by \(\tau\), also  \(M\cap X_2 = (M\cap X_1)^{\tau}\geqslant N_1^\tau =  N_2\), and therefore \(L\leqslant M\). It follows that \(L\) is the unique minimal normal subgroup of \(G\). Thus part (a) is proved.

      Now suppose also that $X$ is quasiprimitive on $V\Delta$. Then $N$ is transitive on $V\Delta$, and hence $L$ has exactly two orbits in $V\Gamma$, namely $\Gamma_0$ and $\Gamma_1$. Hence \(\Gamma\) is $G$-vertex-bi-quasiprimitive. Part (b) now follows from Proposition~\ref{thm: bi-quasi con}.
\end{proof}

\medskip
In Section~\ref{Sec:Ex} we apply Proposition~\ref{cor:biquasi} to give examples of  $G$-vertex-bi-quasiprimitive $(G,2s)$-arc-transitive digraphs for arbitrarily large $s$, such that Lemma~\ref{l:i-or-iia}(b) holds for the $G^+$-action on a bipartite half (see also Remark~\ref{r:Biquasi}(d)). We finish this section by applying Proposition~\ref{cor:biquasi} to construct  $G$-vertex-bi-quasiprimitive $(G,4)$-arc-transitive digraphs such that the bipartite halves are products (as in Construction~\ref{ex:PA}) of two copies of an \((X,2)\)-arc-transitive \(X\)-vertex-primitive digraph of type \(SD\). These digraphs of type \(SD\) were constructed in \cite[Construction 3.1]{quasi}, and some  details about them are given in Remark~\ref{r:sd}. In particular, since each such digraph $\Delta$ is $X$-vertex-primitive, each nontrivial normal subgroup of $X$ is vertex-transitive. Hence $\Delta$ is connected and not bipartite, and also has no cyclic normal quotients so $(\Delta, X, 2)\in\DD$.

\begin{proposition}\label{p:SD}
    Let $(\Delta, X, 2)\in\DD$ with $\Delta$ as in  \cite[Construction 3.1]{quasi} and $X=\Aut(\Delta)$ vertex-primitive of type \(SD\), and let $\Gamma, G$ be the digraph and group given by Construction~\ref{con1}. Then $\Gamma$ is a $(G,4)$-arc-transitive $G$-vertex-bi-quasiprimitive digraph, and is not $(G,5)$-arc-transitive. Also $(\Gamma, G, 4)\in\DD$,  each bipartite half of  $\Gamma$ is a direct product $\Gamma'=\Delta\times\Delta$ as in Construction~\ref{ex:PA}, and the normal quotient operation in Proposition~\ref{Biquasi1} applied to $(\Gamma', X\times X, 2)\in\DD$ produces $(\Delta, X, 2)\in\DD$.  
\end{proposition}

\begin{proof}
    Since $\Delta$ is connected and not bipartite, it follows from Proposition~\ref{thm: bi-quasi con}(a) that $\Gamma$ is connected, and by Proposition~\ref{cor:biquasi},  $\Gamma$ is a  $G$-vertex-bi-quasiprimitive $(G,4)$-arc-transitive digraph, and $G$ has a unique minimal normal subgroup $N\times N$, where $N=\Soc(X)$. Since $\Gamma$ is $G$-vertex-bi-quasiprimitive, it follows  that $N\times N$ has exactly two orbits in $V\Gamma$.  In particular $\Gamma$ has no $G$-normal quotient $\C_r^\to$ for any $r\geqslant3$, and therefore $(\Gamma, G, 4)\in\DD$. By Proposition~\ref{thm: bi-quasi con}(a), each bipartite half of $\Gamma$ is the direct product  $\Gamma'=\Delta\times\Delta$ as in Construction~\ref{ex:PA} admitting the group $G^+=X\times X$. Since $G^+$ has precisely two minimal normal subgroups, each of which is isomorphic to $\Soc(X)$ and is vertex-intransitive on $\Gamma'$ (Proposition~\ref{cor:biquasi}(a)), it follows that $(\Gamma', X\times X, 2)\in\DD$ and Remark~\ref{r:Biquasi}(d) holds for $G$. The normal quotient $(\Gamma')_C$ constructed in Proposition~\ref{Biquasi1} is isomorphic to $\Delta$ and admits the vertex-primitive group $X$ of type \(SD\), so this normal quotient operation produces $(\Delta, X, 2)\in\DD$.  
    
    Finally suppose that $\Gamma$ is $(G,5)$-arc-transitive. Since $\Gamma$ is $G$-vertex-transitive, this property holds if and only if, for $u_0=(v_0,w_0,0)\in\Gamma_0$, $G_{u_0}$ is transitive on $5$-arcs of the form $\gamma:=(u_0,u_1,\dots, u_5)$. Given such a $5$-arc  we may extend it to a $6$-arc $\gamma^+:=(u_0,u_1,\dots, u_5, u_6)$, and by Proposition~\ref{thm: bi-quasi con}(c),  $u_{2i}=(v_i,w_i,0)$ (for $0\leq i\leq 3$) and $u_{2i+1}=(v_i,w_{i+1},1)$ (for $0\leq i\leq 2$) such that $(v_0,v_1,v_2,v_3)$ and $(w_0,w_1,w_2,w_3)$ are $3$-arcs of $\Delta$.    Similarly let $(\gamma')^+:=(u_0',u_1',\dots, u_6')$ be a second such $6$-arc of $\Gamma$ with $u_0'=u_0$ and let $\gamma':=(u_0',u_1',\dots, u_5')$, so  $u_{2i}'=(v_i',w_i',0)$ (for $0\leq i\leq 3$) and $u_{2i+1}'=(v_i',w_{i+1}',1)$ (for $0\leq i\leq 2$) such that $(v_0',v_1',v_2',v_3')$ and $(w_0',w_1',w_2',w_3')$ are $3$-arcs of $\Delta$. Since $G_{u_0}$ is transitive on $5$-arcs with initial vertex $u_0$, there exists an element $(x,y)\in G_{u_0}<G^+=X\times X$ that maps $\gamma$ to $\gamma'$. It follows that $u_i^{(x,y)}=u_i'$ for each $i=0,\dots,5$ and hence in particular that $(w_0,w_1,w_2,w_3)^y= (w_0',w_1',w_2',w_3')$. Since each $3$-arc $(w_0',w_1',w_2',w_3')$ of $\Delta$ with $w_0'=w_0$ arises in these $5$-arcs of $\Gamma$, we conclude that $X_{w_0}$ is transitive on the $3$-arcs of $\Delta$ with initial vertex $w_0$. Then since $\Delta$ is $X$-vertex-transitive, the digraph $\Delta$ is $(X,3)$-arc-transitive, contradicting \cite[Theorem 1.2]{quasi}. Thus $\Gamma$  is not $(G,5)$-arc-transitive, and the proof is complete.
\end{proof}

\section{Proof of Theorem~\ref{t:main}}\label{proof of t:main}

Let $\DD$ be as in \eqref{e:family}, and suppose that  $(\Gamma, G, s):=(\Gamma^{(1)},G^{(1)},s^{(1)})\in\DD$ with $s\geqslant3$ if $\Gamma$ is not bipartite and $s\geqslant6$ if $\Gamma$ is bipartite, so $s=s^{(1)}$ is as in Table~\ref{tab:main}. We consider possible reduction steps as in Definition~\ref{d:reln}. Suppose that one of the triples considered is $(\Gamma^{(i)},G^{(i)},s^{(i)})\in\DD$, for some positive integer $i$.

\smallskip\noindent
\emph{Claim 1: If $G^{(i)}$ is vertex-quasiprimitive on  $V\Gamma^{(i)}$ and $s^{(i)}\geqslant3$, then either  $n=i$ with $G^{(i)}$ of type AS, or (ii) $n=i+1$ with $G^{(i)}$ of type PA and $s^{(n)} =s^{(i)}-1$. In either case $\Soc(G^{(n)})$ is a composition factor of $G^{(i)}$. }

In this case, $G^{(i)}$ has type AS or PA by Proposition~\ref{p:pa}(a). In the former case we complete the reduction sequence at this stage with $n=i$. Suppose then that $G^{(i)}$ has type PA. Then we apply Proposition~\ref{p:pa}(b) and make one more reduction step to the $G^{(i)}$-subnormal quotient digraph $\Gamma^{(i+1)}$ described there admitting the quasiprimitive almost simple component $G^{(i+1)}$ of $G^{(i)}$, obtaining the triple $(\Gamma^{(i+1)},G^{(i+1)},s^{(i+1)})\in\DD$ with  $s^{(i+1)}=s^{(i)}-1$. We complete the reduction sequence at this stage with $n=i+1$, and Claim 1 is proved.

\smallskip
We now complete the proof in the non-bipartite case.

\smallskip\noindent
\emph{Claim 2: The claims of Theorem~\ref{t:main} are valid if $\Gamma$ is not bipartite.} 

Suppose first that $G$ is biquasiprimitive on $V\Gamma$.

If $G=G^{(1)}$ is quasiprimitive on $V\Gamma$, then the assertions of  Theorem~\ref{t:main} follow from Claim 1 with $i=1$, giving $n\leqslant 2$, $s^{(n)} \geqslant s-1= s^{(1)}-1$, and $\Soc(G^{(n)})$ a composition factor of $G^{(1)}$. Assume now that $G$ is not quasiprimitive, so there exists a nontrivial vertex-intransitive normal subgroup $N$. Then $N$ has at least three orbits since $\Gamma$ is not bipartite. By Proposition~\ref{p:quot2}(c), there exists an intransitive normal subgroup $M$ such that $N\leqslant M\unlhd G$ with more than 3 vertex orbits and we can make a reduction with the $G$-normal quotient $\Gamma_M$ to    $(\Gamma^{(2)},G^{(2)},s^{(2)}) := (\Gamma_M,G/M,s)\in\DD$ with $G/M$ vertex-quasiprimitive.  The assertions of  Theorem~\ref{t:main} follow from Claim 1 with $i=2$, giving $n\leqslant 3$, $s^{(n)} \geqslant s^{(2)}-1=s-1$, and $\Soc(G^{(n)})$ a composition factor of $G^{(2)}=G/M$ and hence a composition factor of $G^{(1)}$. This proves Claim 2.

\smallskip
From now on we assume that $\Gamma$ is bipartite so $s\geqslant6$.

\smallskip\noindent
\emph{Claim 3: The claims of Theorem~\ref{t:main} are valid if $\Gamma$ is bipartite.} 

If $G$ is not biquasiprimitive on $V\Gamma$ then by  Proposition~\ref{p:quot2}(c), there exists a nontrivial normal subgroup $M$ of $G$ with at least three vertex orbits such that  we can make a reduction step with the $G$-normal quotient $\Gamma_M$ to    $(\Gamma^{(2)},G^{(2)},s^{(2)}) := (\Gamma_M,G/M,s)\in\DD$ with $G/M$ vertex-bi-quasiprimitive.  Thus we may assume that $(\Gamma^{(j)},G^{(j)},s^{(j)})\in\DD$ with $j\leqslant2$, $G^{(j)}$ vertex-bi-quasiprimitive and a quotient of $G^{(1)}$, and  $s^{(j)}=s\geqslant6$. We use the notation from Proposition~\ref{Biquasi0} and Remark~\ref{r:Biquasi} for the bipartite half $\Gamma'$ of $\Gamma^{(j)}$ with vertex set $\Delta$ and group $H$ where $(G^{(j)})^+=\Diag_\varphi(H\times H)$.  We make a reduction step using the bipartite half to $(\Gamma^{(j+1)},G^{(j+1)},s^{(j+1)}) := (\Gamma', H, \lfloor s/2\rfloor)\in\DD$. 
In the case where Remark~\ref{r:Biquasi}(c) holds, the group $H$ is quasiprimitive on $\Delta$ of type AS or PA, by Lemma~\ref{l:i-or-iia}.  
 The assertions of  Theorem~\ref{t:main} follow from Claim 1 with $i=j+1$, giving $n\leqslant j+2\leqslant 4$, $s^{(n)} \geqslant s^{(j+1)}-1=\lfloor s/2\rfloor -1$, and $\Soc(G^{(n)})$ a composition factor of $G^{(j+1)}=H$ and hence a composition factor of $G^{(1)}$.
Thus we may assume that Remark~\ref{r:Biquasi}(d) holds for $H$. Then we make a further reduction step using the $H$-normal quotient $(\Gamma^{(j+1)})_C$ in Proposition~\ref{Biquasi1} to $(\Gamma^{(j+2)},G^{(j+2)},s^{(j+2)}) :=((\Gamma^{(j+1)})_C,H/C,s^{(j+1)})\in\DD$ with $H/C$ vertex-quasiprimitive of type AS or PA. The assertions of  Theorem~\ref{t:main} follow from Claim 1 with $i=j+2$, giving $n\leqslant j+3\leqslant5$,  $s^{(n)} \geqslant s^{(j+2)}-1=\lfloor s/2\rfloor -1$, and $\Soc(G^{(n)})$ a composition factor of $G^{(j+2)}=H/C$ and hence a composition factor of $G^{(1)}$. This proves Claim 3, and completes the proof of Theorem~\ref{t:main}.

\section{Quasiprimitive and bi-quasiprimitive digraphs admitting $\Alt(n)$ or $\Sym(n)$}\label{Sec:Ex}

In this final section we explore some of the properties of the $s$-arc-transitive digraphs constructed in \cite{CLP} from finite alternating and symmetric groups. We show that some of these digraphs are vertex-quasiprimitive, while others are vertex-bi-quasiprimitive, and we explore how Construction~\ref{con1} can be applied in the former case. In fact several families of digraphs are constructed in \cite{CLP}, and we examine  those constructed in \cite[Proof of Theorem 1]{CLP}.  First we give the definition of this family in Construction~\ref{ex: alt and sym}, see \cite[pp 66-67]{CLP}.

\begin{construction}\label{ex: alt and sym}
Let $k, s, n$ be positive integers such that $k\geqslant2$ and $\gcd(n,s+1)=1$. 
\begin{enumerate}
\item[(a)] Define the permutations $x_1,\dots,x_{s+1},a\in \Sym((s+1)k+n)$ as follows.
\begin{align*}
    x_{i}&= ((i-1)k+1,(i-1)k+2,\ldots, ik) & \mbox{for  \(1\leqslant i\leqslant s+1\)}\\
    y_1&= (1,k+1,\ldots, sk+1,(s+1)k+1, (s+1)k+2, \ldots,(s+1)k+n) & \\
    y_i &= (i,k+i,\ldots,sk+i) & \mbox{for  \(2\leqslant i\leqslant k\)}\\
    a&= y_1\,y_2\,\dots \,y_k  & 
\end{align*}

\item[(b)] 
Let \(H=\langle x_1,\ldots,x_{s+1}\rangle\) and \(X=\langle H,a\rangle\).

\item[(c)] Let  \(\Delta=\Delta(k,s,n)\) be the digraph with \(V\Delta=\{Hx\mid x\in X\}\) and \(A\Delta=\{(Hx,Hax)\mid x \in X\}\). 
\end{enumerate}
    
\end{construction}

We note the following property proved in \cite{CLP}.

\begin{lemma}\cite[pp 66-67]{CLP}\label{ex: alt and sym 1}
    Suppose that Construction \ref{ex: alt and sym} holds. Then \(\Delta\) is a connected \(X\)-vertex-transitive \((X,s)\)-arc-transitive digraph of valency \(k\), and 
    \[
    X=\begin{cases}
    \Alt((s+1)k+n)&\text{if \(k\) is odd and \(s+n\) is even}\\
    \Sym((s+1)k+n)&\text{otherwise}
\end{cases}
\]
\end{lemma}
Although the graphs $\Delta(k,s,n)$ can be defined without the condition $\gcd(n,s+1)=1$, this condition is essential for the proof given in \cite[Proof of Theorem 1]{CLP} of Lemma~\ref{ex: alt and sym 1}. Moreover the conclusion of Lemma~\ref{ex: alt and sym 1} is not necessarily true if $\gcd(n,s+1)>1$. For example, if $k=s=2$ and $n=3$, then  $\gcd(n,s+1)=3$,  $(s+1)k+n=9$, and $X= S_3\wr A_3$, preserving the partition $\{1,2,7\mid 3,4,8\mid 5,6,9\}$. 
Next we determine when $\Delta$ is bipartite and when $X$ is vertex-quasiprimitive.

\begin{table}[]
    \centering
    \begin{tabular}{c|cccccc}
     Line  &\(k\)  &  \(n\)&\(s\)&\(X\)& $X$-action on \(V\Delta\)& $\Delta$ bipartite?\\  \hline
       1&even&even&even&\(\Sym\)&quasiprimitive&no\\
       2&even&odd&--&\(\Sym\)&quasiprimitive&no\\
       3& odd & odd&odd&\(\Alt\)& quasiprimitive & no \\
       4&odd & even & even &\(\Alt\)&quasiprimitive& no \\
       5& odd & odd &even &\(\Sym\)&bi-quasiprimitive& yes \\ \hline
    \end{tabular}
    \caption{Parities and properties for Proposition~\ref{ex: alt and sym 2}}
    \label{tab:placeholder}
\end{table}

\begin{proposition}\label{ex: alt and sym 2} 
   Suppose that Construction \ref{ex: alt and sym} holds. Then precisely one of the lines of Table~\ref{tab:placeholder} holds, where an entry `Alt' or `Sym' in column $X$ indicates that $X= \Alt((s+1)k+n)$ or $\Sym((s+1)k+n)$ respectively.  Moreover $(\Delta, X, s)\in\DD$, and:
       \begin{enumerate}
    \item \(X\) acts bi-quasiprimitively on \(V\Delta\) if and only if \(kn\) is odd and \(s\) is even; 
    \item for each even \(s\) and odd $k>1$, 
    \begin{enumerate}
        \item[(i)]     there are infinitely many integers \(n\) such that \((\Delta, X, s)\in\DD\) with $\Delta$ of valency $k$ and \(X\)-vertex-bi-quasiprimitive;
        \item[(ii)]     there are infinitely many integers \(n\) such that \((\Delta, X, s)\in\DD\) with $\Delta$ of valency $k$ and \(X\)-vertex-quasiprimitive.
    \end{enumerate}
    \end{enumerate}
\end{proposition}

\begin{proof}
The entries in columns $k, n, s, X$ follow immediately from Lemma~\ref{ex: alt and sym 1}, noting  that, if $n$ is even then $s+1$ must be odd since $\gcd(n,s+1)=1$, and hence $s$ must also be even.   Next we consider the $X$-action on $V\Delta$. Note that $X$ is quasiprimitive on $V\Delta$ if and only if (the unique minimal normal subgroup of $X$)  $\Alt((s+1)k+n)$ is transitive on $V\Delta$. In this case $\Delta$ is not bipartite since $\Alt((s+1)k+n)$ has no index $2$ subgroup preserving a vertex-bipartition, and similarly $X$ has no intransitive normal subgroup of index $r\geq3$ yielding a $X$-normal quotient $C_r^\to$. Thus $(\Delta, X, s)\in\DD$ in this case and in particular the entries in lines 3 and 4 of Table~\ref{tab:placeholder} are correct.

Assume now that $X= \Sym((s+1)k+n)$, so $X'=\Alt((s+1)k+n)$ is an index $2$ subgroup.  Consider the vertex stabiliser \(H=\langle x_1,\ldots,x_{s+1}\rangle\), and note that each generator $x_i$ is a $k$-cycle. Suppose first that $k$ is even. Then each $x_i$ is an odd permutation and hence $H\cap X'$ has index $2$ in $H$. Thus $|V\Delta|=|X:H|=|X':H\cap X'|$ and hence $X'$ is vertex-transitive, so $X$ is quasiprimitive on $V\Delta$. The argument of the previous paragraph shows that $\Delta$ is not bipartite and $(\Delta, X, s)\in\DD$ in this case. In particular the entries in lines 1 and 2 of  Table~\ref{tab:placeholder} are correct. 
Suppose now that $k$ is odd. Since $X= \Sym((s+1)k+n)$, the integers $n$ and $s$ have different parities by Lemma~\ref{ex: alt and sym 1}, and since $\gcd(n,s+1)=1$ it follows that $n$ is odd and $s$ is even. Now each $x_i$ is an even permutation since $k$ is odd, and hence $H< X'<X$ so $X'$ has two orbits $\Delta, \Delta'$ of equal size in $V\Delta$. As $\Delta$ is connected it follows that $\Delta$ is bipartite. Since $X'$ is the setwise stabiliser of $\Delta$, and since $(X')^\Delta\cong \Alt((s+1)k+n)$, a nonabelian simple group, it follows that $X'$ acts quasiprimitively on $\Delta$, and similarly also on $\Delta'$. Hence \(\Delta\) is \(X\)-vertex-bi-quasiprimitive, and $X$ has no intransitive normal subgroup of index $r\geq3$ yielding a $X$-normal quotient $C_r^\to$. Thus $(\Delta, X, s)\in\DD$, and the entries in line 5 of  Table~\ref{tab:placeholder} are correct. Hence the first assertions, and also part (1) are proved.

Finally we prove part (2).  Suppose that $s$ is even and that $k$ is odd ($k>1$). Then there are infinitely many primes larger than $s+1$. First let us take $n$ to be any of these primes. Then $kn$ is odd and $\gcd(n,s+1)=1$. It follows from part (1) that \(\Delta\) is \(X\)-vertex-bi-quasiprimitive,  and it follows from Lemma~\ref{ex: alt and sym 1} that  \(\Delta\) is a connected  \((X,s)\)-arc-transitive digraph of valency $k$, so \((\Delta, X, s)\in\DD\). Thus part (2)(i) is proved. Alternatively take $n=2p$, where $p$ is any prime larger than $s+1$. Then again $\gcd(n,s+1)=1$, and this time $k$ is odd and $s+n$ is even. Thus by  Lemma~\ref{ex: alt and sym 1},  \(\Delta\) is a connected \(X\)-vertex-transitive, \((X,s)\)-arc-transitive digraph of valency $k$, and by line 4 of Table~\ref{tab:placeholder}, $X$ is quasiprimitive on $V\Delta$.  Thus also part (2)(ii) is proved. 
\end{proof}

Finally we summarise our knowledge about vertex-quasiprimitive and vertex-bi-quasiprimitive highly-arc-transitive digraphs associated with the digraphs from Construction~\ref{ex: alt and sym}.

\begin{proposition}\label{ex alt and sym 3}
Suppose that Construction \ref{ex: alt and sym} holds for $\Delta = \Delta(k,s,n)$ and $X$. 
\begin{enumerate}
    \item[(a)] If one of lines $1$--$4$  of Table~\ref{tab:placeholder} holds for $k,n,s$, then  $(\Delta, X, s)\in\DD$ with $\Delta$ of valency $k$ and $X$-vertex-quasiprimitive of type $AS$. Also for any positive integer $m>1$, the product $\Delta^m$ as in Construction~\ref{ex:PA} yields  $(\Delta^m, X\wr \Sym(m),s)\in\DD$ with $\Delta^m$ of valency $k^m$ and $(X\wr \Sym(m))$-vertex-quasiprimitive of type $PA$.

    \item[(b)] If one of lines $1$--$4$  of Table~\ref{tab:placeholder} holds for $k,n,s$ so $\Delta$ is $X$-vertex-quasiprimitive of valency $k$ and not bipartite, then the digraph, group $\Gamma, G$ arising from Construction~\ref{con1} applied to either $\Delta, X, s$, or $\Delta^m, X\wr \Sym(m),s$ for $m>1$, yields $(\Gamma,G,2s)\in\DD$ with $\Gamma$ being $G$-vertex-biquasiprimitive as in Lemma~\ref{l:i-or-iia}(b), and of valency $k$ or $k^m$ respectively.

    \item[(c)] If line $5$ of Table~\ref{tab:placeholder} holds for $k,n,s$  with $s\geq4$, then 
    $\Delta$ is $X$-vertex-biquasiprimitive of valency $k$ with $X=\Sym((s+1)k+n)$, and the bipartite half $\Delta'$ has valency $k^2$ and is $X'$-vertex-quasiprimitive of  type $AS$ as in Lemma~\ref{l:i-or-iia}(a) yielding $(\Delta', X',s/2)\in\DD$.
    
\end{enumerate}
    
\end{proposition}

\begin{proof}
Note that in all cases we have $(\Delta, X, s)\in\DD$ of valency $k$ by Proposition~\ref{ex: alt and sym 2}. 
Suppose first that one of lines $1$--$4$  of Table~\ref{tab:placeholder} holds for $k,n,s$, so 
$\Delta$ is $X$-vertex-quasiprimitive of type $AS$ and is not bipartite (Proposition~\ref{ex: alt and sym 2}). By Proposition~\ref{cor:biquasi}, for $\Gamma, G$ arising from Construction~\ref{con1} applied to $\Delta, X, s$, we have $(\Gamma,G,2s)\in\DD$ with $\Gamma$ being $G$-vertex-biquasiprimitive. The fact that $G$ is as in Lemma~\ref{l:i-or-iia}(b) follows since $G^+=X\times X$ has two minimal normal subgroups isomorphic to $\Soc(X)$ and interchanged by $G$. Also $\Gamma$ has valency $k$ by Proposition~\ref{thm: bi-quasi con}.

Now let $m>1$ and consider the product $\Delta^m$ as in Construction~\ref{ex:PA}. As discussed in Construction~\ref{ex:PA}, $\Delta^m$ has valency $k^m$, is $(X\wr \Sym(m))$-vertex-quasiprimitive of type $PA$, and is $(X^m,s)$-arc-transitive. Thus each nontrivial normal subgroup of $X\wr \Sym(m)$ is vertex-transitive and hence $\Delta^m$ is not bipartite and has no cyclic $(X\wr \Sym(m))$-normal quotient $C_r^\to$ for any $r\geq 3$. 
Hence  $(\Delta^m, X\wr \Sym(m),s)\in\DD$, and the proof of part (a) is complete. By Proposition~\ref{cor:biquasi}, for $\Gamma, G$ arising from Construction~\ref{con1} applied to $\Delta^m, X\wr \Sym(m),s$, we have $(\Gamma,G,2s)\in\DD$ with $\Gamma$ being $G$-vertex-biquasiprimitive. The fact that $G$ is as in Lemma~\ref{l:i-or-iia}(b) follows since $G^+=(X\wr \Sym(m))\times (X\wr \Sym(m))$ (by Construction~\ref{con1}), and $G^+$ has two minimal normal subgroups isomorphic to $\Soc(X)^m$ and interchanged by $G$. Also $\Gamma$ has valency $k$ by Proposition~\ref{thm: bi-quasi con}, and this completes the proof of part (b).

Suppose finally that line $5$ of Table~\ref{tab:placeholder} holds, so $\Delta$ has valency $k$ and is $G$-vertex-biquasiprimitive with $G=\Sym((s+1)k+n)$ by Proposition~\ref{ex: alt and sym 2}. Then by Proposition~\ref{p:bip}, for the bipartite half $\Delta'$, we have  $(\Delta', G',s/2)\in\DD$ with $G'=G^+=\Alt((s+1)k+n)$. Since $s\geq4$, it follows that $\Delta'$ has valency $k^2$, and part (c) is proved.
\end{proof}

\end{document}